\documentclass[11pt,reqno]{amsart}
\usepackage{amsbsy,amsfonts,amsmath,amssymb,amscd,amsthm}
\usepackage{mathrsfs}
\usepackage{subfigure,enumitem,color,graphicx}
\usepackage[a4paper,text={6.1in,8in},centering,includefoot,foot=1in]{geometry}
\usepackage{pxfonts}
\usepackage[colorlinks=true,linkcolor=blue,citecolor=green]{hyperref}
\usepackage{srcltx}

\small\normalsize
\newtheorem{theorem}{Theorem}[section]
\newtheorem*{theorem A}{Theorem A}
\newtheorem*{theorem B}{Theorem B}
\newtheorem*{theorem C}{Theorem C}

\newtheorem{lemma}[theorem]{Lemma}

\newtheorem{definition}[theorem]{Definition}
\newtheorem{corollary}[theorem]{Corollary}
\newtheorem{remark}[theorem]{Remark}

\newtheorem{proposition}[theorem]{Proposition}

\let\oldtocsection=\tocsection
\let\oldtocsubsection=\tocsubsection
\renewcommand{\tocsection}[2]{\hspace{0em}\bf\oldtocsection{#1}{#2}}
\renewcommand{\tocsubsection}[2]{\hspace{1em}\oldtocsubsection{#1}{#2}}
\let\oldtocsubsubsection=\tocsubsubsection
\renewcommand{\tocsubsubsection}[2]{\hspace{2em}\oldtocsubsubsection{#1}{#2}}
\begin{document}
\author[Ehsani]{A. Ehsani$^1$}
\address{\centerline{$^{1,3,4}$ Department of Mathematics, Ferdowsi University of Mashhad, }
\centerline{Mashhad, Iran.}  }
\email{aza\_ehsani@yahoo.com}
\email{ghane@math.um.ac.ir}
\email{zaj.marzie@yahoo.com}
\author[Fakari]{A. Fakhari$^2$}
\address{\footnotesize \centerline{$^2$Department of Mathematics,} \centerline{Shahid
Beheshti University, G.C., Tehran 19839, Iran}}
\email{a\_fakhari@sbu.ac.ir}
\author[Ghane]{F. H. Ghane$^3$}
\author[Za]{M. Zaj$^4$}

 \thanks{$^*$Corresponding author}
\title[ SRB measures for certain class of non-uniformly hyperbolic endomorphisms] { SRB measures for certain class of non-uniformly hyperbolic endomorphisms on the solid torus}
 \maketitle
\begin{abstract}
In this paper we address the existence and ergodicity of non-hyperbolic attracting sets for a
certain class of smooth endomorphisms on the solid torus. Such systems allow a formulation as a skew product system defined by planar diffeomorphisms have contraction on average which forced by any expanding circle map. These attractors are invariant graphs of upper semicontinuous maps which support exactly one $SRB$ measure. In our approach, these skew product systems arising from iterated function systems generated
by a finitely many weak contractive diffeomorphisms. Under some conditions including negative fiber Lyapunov exponents, we prove the existence of unique non-hyperbolic attracting invariant graphs for these systems which attract positive orbits of almost all initial points. Also, we prove that these systems are Bernoulli and therefore they are mixing.
Moreover, these properties remain true under small perturbations in the space of endomorphisms on the solid torus.
\end{abstract}
\section{Introduction and statement of results}
The main goal of this paper is to discuss non-hyperbolic attractors which are invariant graph and carry SRB measures. The existence of SRB measures was established for hyperbolic systems \cite{Bo, BR, Ru, Si} and then extended in the case of partial hyperbolicity \cite{ABV, BV, Ca, Ts}.
In the case of surface endomorphism, Tsujii \cite{Ts} proved that a generic partial hyperbolic endomorphism admits
finitely many physical measures whose basins have full Lebesgue measure. In higher dimensional case, Volk \cite{V} presented an
open set of smooth endomorphisms such that any of them had a transitive hyperbolic attractor
with a non-empty interior supporting a unique SRB measure.
In \cite{St2} Stark provides the conditions for the existence and regularity of invariant graphs and discuss a number of applications to the filtering of time series, to synchronization and to quasiperiodically forced systems. 
An invariant graph that attracts almost surely orbits is the natural generalization of a stable fixed point to the case of forced systems.
Attracting invariant graphs have a wide variety of applications in many branches of nonlinear dynamics (e.g. \cite{CD, DC, HOJ1, HOJ2, PC, SD, T} etc.).
Here we focus on skew product systems.
In skew product systems with uniformly contracting fiber maps, it is easy to see
that there exist invariant attracting sets for the overall dynamics, which are the graph of
continuous functions (see \cite{HP}, Theorem 6.1a, \cite{HPS}). In the nonuniform case, when the fiber map is contracting on average \cite{AC1, E1, E2}
(for instance when the skew product map possesses a negative maximal Lyapunov exponent in the fibre),
less is known about the stability of the dynamics under additive noise.
In \cite{BHN}, the authors have proved that this kind of systems preserve ergodicity and higher order mixing properties under deterministic perturbation of the fiber map and perturbation by
i.i.d. additive noise. Furthermore they have shown that the invariant measure for the perturbed system is continuous in the Hutchinson metric.
Here we extend this result to a certain class of endomorphisms. In fact, the aim of this paper is to discuss a class of smooth endomorphisms on the solid torus admitting robust non-uniformly hyperbolic attracting sets which are invariant graph and supporting a unique SRB measure.
To establish this result, first we prove it for a special class of skew products over expanding circle maps
with weak contractive planar fiber maps. In our approach we also prove the occurrence of a master-slave synchronization. A \emph{master-slave synchronization} means the convergence of orbits starting at different initial
points when iterated by the same sequence of diffeomorphisms
which is explained by a single attracting invariant graph for the skew
product system \cite{St1, St2}.
We also show that the attractors are the support of unique invariant ergodic SRB measures

Throughout this paper we assume that $X$ is a compact ball of $\mathbb{R}^2$ and $S^1$ is the unit circle. Take the solid torus $\mathbb{T}=S^1 \times X$.
Denote by $\mathcal{C}(\mathbb{T})$ the space of all skew products over $\varphi$ with the fiber $X$, where $\varphi$ is a linear expanding map on the circle, i.e. the maps of the form
$$F : (t,x) \mapsto (\varphi(t),f_t(x)), \ t \in S^1, \ x \in X.$$
Here $f_t$ is a diffeomorphism onto its image, depending $C^0$-continuously to the base parameter $t$. The metric on $\mathcal{C}(\mathbb{T})$ is defined as
\begin{equation}\label{e000}
 \textnormal{dist}(F, \widetilde{F}):= \sup_{t} \textnormal{dist}_{C^1}(f_t^{\pm 1}, \widetilde{f}_t^{\pm 1}).
\end{equation}
In this article, we generalize the concept of a bony graph \cite{KV} to our setting. We say that a closed invariant set of a skew product $F$ is a {\it bony graph}  if it intersects almost every fiber at a single point, and any other fiber at a compact connected set which is called a bone. A bony graph can be represented as a disjoint union of two sets, $K$ and $\Gamma$, where $K$ denotes the union of the
bones. The projection of $K$ by the natural projection map to the base has zero measure, while $\Gamma$ is the graph of some
measurable function from base to the fiber.
Let $\Delta$ be a maximal attractor of $F$. We say that $\Delta$ is
 a \emph{continuous-bony graph} (CBG) if $\Delta$ is a bony graph and the graph function is upper semicontinuous.
\begin{theorem A}
There exists a nonempty open set in $\mathcal{C}(\mathbb{T})$ such that any skew product $F$
belonging to it, admits a non-hyperbolic attractor $\Delta_F$ such that
\begin{enumerate}
\item $\Delta_F$ is a continuous bony graph,
\item $\Delta_F$ has a negative maximal Lyapunov exponent,
\item $\Delta_F$ is supporting a unique SRB measure.
\end{enumerate}
In particular, it is Bernoulli and so mixing.
\end{theorem A}
As we mentioned before, when the maps are uniformly contracting in the fiber direction it is extremely easy to prove the existence of a continuous invariant graph supporting an ergodic SRB measure.
In most cases it is difficult to get
sharp uniform bounds for the contraction rate of fiber maps, instead, this is usually expressed in terms of the most negative
Lyapunov exponent, which is an averaged quantity. It is important to understand the regularity
properties of graph function when the contraction in the fiber is given in terms of Lyapunov exponents.
Here we consider the nonuniform case
and construct Sinai-Ruelle-Bowen (SRB) measures supported on
non-uniformly hyperbolic attractors of smooth endomorphisms and show that these attractors have the form of an invariant graph.
In particular, the graph function is upper semicontinuous.

Denote by $\mathcal{E}(\mathbb{T})_k$ the
space of all $C^1$-smooth $k$-to-1 coverings (endomorphisms) of $\mathbb{T}$ by itself, for some suitable $k >2$, with $C^1$-topology.
\begin{theorem B}
There exists a nonempty open set in $\mathcal{E}(\mathbb{T})_k$ such that any  $\mathcal{F}$ belonging to it
admits a non-uniformly hyperbolic attractor which is a continuous bony invariant graph with negative maximal fiber Lyapunov exponent. Moreover, there exists a unique $SRB$ measure supported on
the graph. In particular, it is Bernoulli and therefore it is mixing.
\end{theorem B}
\textbf{Organization of the paper.}~Section 2 is devoted to the construction of random iterated function system cited throughout the article. First, some classical notions in the theory of iterated function systems are provided. In the first and second steps which are handled in the next sub section, we introduce a weakly hyperbolic iterated function systems which is the essential step in our construction of random iterated function systems, introduced explicitly in the last sub section. In section 3, we study topological properties of the attractor of the random iterated function systems. In particular, we show that the random iterated function systems admits a non-hyperbolic topological attractor with nonempty interior. Section 4 deals with the ergodic properties. We show the existence and uniqueness of SRB measures supporting on the unique attractor.
\section{Random Iterated Function Ssytem}
This section is devoted to the construction of the random iterated function systems applied in the main theorems of this article. First, we recall some common notions and definition in the theory of iterated function systems. Next, defining a single diffeomorphism, we introduce an iterated function systems applied in the construction of desired random iterated function system.
\subsection{Iterated Function Systems}
This section is devoted to study
 a certain class of non hyperbolic iterated function systems admitting compact connected attractors with nonempty interiors. Moreover, these attractors are the support of a unique SRB measure.
To state the main result of this section, we need to introduce some notations and recall several background definitions and concepts.

An iterated function system is the action of a semigroup generated by a
family of maps with a fixed distribution from which a map is chosen, independently at each
iterate. To be more precise, let  $\Lambda$ and $X$ be compact metric spaces and $\mathcal{F}=\{f_\lambda : \lambda \in \Lambda \}$ be a family of homeomorphisms on $X$.
The space $\Lambda$ is called the \emph{parameter space} and $X$ is called the \emph{fiber}.
The space $\Lambda^{\mathbb{N}}$ of infinite words with alphabet in $\Lambda$, endowed with the product topology, will be denoted by $\Omega^+ := \Lambda^{\mathbb{N}}$.  For each $k \in \mathbb{N}$, we set
$$\mathcal{F}^k := \mathcal{F}^{k-1} \circ \mathcal{F}, \ \mathcal{F}^0 := \{Id\},$$
where $\mathcal{F}^{k-1} \circ \mathcal{F}=\{f \circ g: f\in \mathcal{F}^{k-1}, \ g\in \mathcal{F} \}$.
Write $\langle \mathcal{F} \rangle^+$ for the semigroup generated by $\mathcal{F}$, that is,
$\langle \mathcal{F} \rangle^+ = \bigcup_{k=0}^\infty \mathcal{F}^k.$ The
action of the semigroup $\langle \mathcal{F} \rangle^+$ is called the \emph{iterated function system} (or IFS) associated to $\mathcal{F}$ and we denote it by IFS$(\mathcal{F})$.
For $x \in M$, we write the \emph{orbits} of the action of this semigroup as
$$\langle \mathcal{F} \rangle^+(x)=\{f(x) : f \in \langle \mathcal{F} \rangle^+ \}.$$
A sequence $\{x_n : n \in \mathbb{N}\}$ is called a \emph{branch of an orbit} of IFS$(\mathcal{F})$
if for each $n \in \mathbb{N}$ there is $f_n \in \mathcal{F}$ such that $x_{n+1}=f_n(x_n)$.
Let $\mathcal{K}(X)$ denote the set of nonempty compact subsets of $X$ endowed with the
Hausdorff metric topology. Then $\mathcal{K}(X)$ is also
a complete metric space and it is compact whenever $X$ is compact.
For an iterated function system IFS$(\mathcal{F})$ we define the associated Hutchinson
operator by $$\mathcal{L} : \mathcal{K}(X) \to \mathcal{K}(X), \ K \mapsto \mathcal{L}(K)= \bigcup_{\lambda \in \Lambda} f_\lambda(K).$$
A set $K \in \mathcal{K}(X)$ is a \emph{strict attractor} of IFS$(\mathcal{F})$ if there exists an open neighbourhood $U(K)\supset K$ such that in the Hausdorff metric
\begin{equation}\label{e21}
 \mathcal{L}^k(B)\to K \ \textnormal{as} \ k \to \infty, \ \textnormal{for} \ U(K)\supset B \in \mathcal{K}(X).
\end{equation}
The basin $B(K)$ of an attractor $K$ is the union of all open neighborhoods $U$ for which $(\ref{e21})$ holds.
We remark that it is usually to include in the definition of attractor that $\mathcal{L}(K) = K$ and $K$ is an $\mathcal{L}$-invariant set. If $U = X$ we say that the IFS$(\mathcal{F})$ possesses a \emph{global attractor}.

From now on, we assume that $(X,d)$ is a compact metric space and IFS$(\mathcal{F})$ is a finitely generated iterated function system on $X$ with generators $\{f_1, \ldots, k\}$.
Let $\Omega^+$ be the symbol space $ \{1,\dots, k \}^\mathbb{N}$ equipped with the product topology.
We consider any probability $\mathbb{P}^+$ on $\Omega^+$
with the following property: there exists $0 < p \leq \frac{1}{k}$ so that $\omega_n$ is selected randomly from $\{1, \ldots, k\}$ in such a way that the
probability of $\omega_n = i$ is greater than or equal to $p$, for
all $i \in \{1, \ldots, k\}$ and $n \in \mathbb{N}$. More formally, in terms of conditional probability,
$$\mathbb{P}^+(\omega_n=i|\omega_{n-1}, \ldots, \omega_1) \geq p.$$
Here, we consider the Bernoulli measure on $\Omega^+$ which is
a typical example of these kinds of probabilities.
Let $\sigma:\Omega^+ \to \Omega^+$ denote the left shift, i.e. $(\sigma \omega)_j = \omega_{j+1}$, for all $\omega \in \Omega^+$ and $j \geq 0$. It is well known that $\sigma$ is an ergodic transformation preserving the probability $\mathbb{P}^+$, see \cite{W2}.

The skew product $\varphi_{\mathcal{F}}$ associated to IFS$(\mathcal{F})$ is defined by
$$\varphi_{\mathcal{F}} : \Omega^+ \times X \to \Omega^+ \times X, \  \varphi_{\mathcal{F}}(\omega,x)=(\sigma \omega, f_{\omega_0}(x)).$$
Putting
$$\textnormal{Lip}_1(X) = \{f : X \to \mathbb{R} : |f(x)- f(y)| \leq d(x, y) \ \textnormal{for} \ \textnormal{all} \ x, y \in X\},$$
define the \emph{Hutchinson metric} on the set $\mathcal{M}(X)$, the space of all Borel probability measures,
by
\begin{equation*}
d_H(\nu,\mu) = \sup \{ |\int_X f d\nu - \int_X f d\mu : f \in \textnormal{Lip}_1(X)|\}.
\end{equation*}
In \cite[Thm.~3.1]{K}, the author proved that for every metric space $X$, the topology $\mathcal{T}$ on $\mathcal{M}(X)$ generated by $d_H(\nu,\mu)$
coincides with the topology $\mathcal{W}$ of weak convergence if and only if $\textnormal{diam}(X)< \infty$.
Moreover, the space $\mathcal{M}(X)$ is complete in the metric $d_H$ if and only if $X$ is complete (see \cite[Thm.~4.2]{K}).
We define the \emph{Transfer Operator} $T : \mathcal{M}(X) \to \mathcal{M}(X)$ by the formula,
$$T(\mu)(B) := \frac{1}{k}\sum_{i=1}^k \mu(f_i^{-1}(B)),$$ for any Borel subset $B$ and for each measure $\mu \in  \mathcal{M}(X)$. If a measure $\mu \in M(X)$ is
a fixed point of the transfer operator we say that $\mu$ is a \emph{stationary measure} (or an \emph{invariant measure}) for IFS$(\mathcal{F})$.
We say that an invariant measure for IFS$(\mathcal{F})$ is \emph{ergodic} if for every continuous function $\phi : X \to \mathbb{R}$, every $x \in X$ and $\mathbb{P}^+$-almost every $\omega \in \Omega^+$  we have
$$\lim_{n\to\infty} \frac{1}{n}\sum _{j=0}^n \phi(f_\omega^j(x))= \int_X \phi d\mu.$$
\subsection{Non-hyperbolic Attractor with Non-empty Interior}
Let $\mathcal{F}=\{f_1, \ldots, f_k\}$ be a
family of Lipschitz maps on X with the map $f_i$ having Lipschitz constant $C_i$. We form the $\textnormal{IFS}(X; f_1, \ldots, f_k: p_1, \ldots, p_k)$ by choosing $f_i$ with probability $p_i$ so that $\sum_{i=1}^k p_i=1$.

The iterated function system IFS$(\mathcal{F})$ is \emph{contractive on average} \cite{BDE} if
\begin{equation*}\label{e11}
  \sum_{i=1}^k p_i \log C_i< 0.
\end{equation*}
\begin{proposition}[\cite{BDE}, \cite{Fr} and \cite{BE}]\label{mainprop}
Let IFS$(\mathcal{F})$ be an IFS on a compact metric space $(X, d)$.
Then a probability measure $\mu \in \mathcal{M}(X)$ is invariant under the IFS$(\mathcal{F})$, i.e. $\mathbb{P}^+ \times \mu  $ is invariant under the skew product $\varphi_{\mathcal{F}}$ associated with IFS$(\mathcal{F})$, if and only if $\mu$ is a fixed point of the Markov operator operator $T$. Moreover, if the IFS$(\mathcal{F})$ is contractive on average then
\begin{itemize}
\item the operator $T$ is contractive with respect to the Hutchinson metric and as a consequence,
IFS$(\mathcal{F})$ has a unique invariant probability measure, $\mu$ say,
\item for any continuous function
 $\phi:X \to \mathbb{R}$ and any $x \in X$, we have
\begin{equation*}
\lim_{n \to + \infty}\frac{1}{n}\sum_{i=0}^{n-1}\phi(f_\omega^i(x))=\int_X \phi(x)d\mu(x), \ \textnormal{for} \ \mathbb{P}^+\textnormal{-}a.e \ \omega \in \Omega^+,
\end{equation*}
where $f_\omega^i:=f_{\omega_i}\circ \ldots \circ f_{\omega_1}$.
\item  if $K$ is the support of $\mu$, then $x \in K$ if and only if for every neighborhood of $x$, almost
all trajectories visit the neighborhood infinitely often.
\end{itemize}
\end{proposition}
We say that $\omega=(\omega_1 \omega_2 \ldots \omega_n \ldots) \in \Omega^+$ is \emph{disjunctive} to mean that, given any $n \in \mathbb{N}$ and any finite word $\theta_1 \ldots \theta_n$ of the alphabets $\{1, \ldots, k\}$,
there is an $L \in \mathbb{N}$ such that $\omega_L \ldots \omega_{L+n-1} = \theta_1 \ldots \theta_n$.
An attractor $K$ of IFS($\mathcal{F}$) is \emph{point-fibred} if for every compact subset
$C \subset B(K)$,
$$\lim_{n \to \infty}f_{\omega_1}\circ \ldots \circ f_{\omega_n}(C) \subset K, $$
is a singleton, independent of $C$ in which, the 
convergence happens in Hausdorff metric.
\begin{corollary}
Let $\textnormal{IFS}(\mathcal{F})$ be a contractive on average iterated function system of homeomorphisms on a compact metric space $X$ with a unique global point-fibred attractor $K$. Then $\textnormal{IFS}(\mathcal{F})$ admits an invariant ergodic measure $\mu$ with $\textnormal{supp}(\mu)=K$.
\end{corollary}
\begin{proof}
By Proposition \ref{mainprop}, IFS($\mathcal{F}$) admits a unique ergodic invariant measure $\mu$. So, it remains to show that $\textnormal{supp}(\mu)=K$. Since IFS($\mathcal{F}$) has a unique global point-fibred attractor $K$ then Theorem 6 of \cite{BV} ensures that for each disjunctive sequence $\omega \in \Omega^+$ and each $x \in X$, the orbital branch $\mathcal{O}(x,\omega)=\{f_\omega^n(x): n\geq 0\}$ is dense in $K$. Now this fact with together the last statement of Proposition \ref{mainprop} imply that $\textnormal{supp}(\mu)=K$.
\end{proof}
Edalat  in \cite{E} defined the notion of weakly hyperbolic iterated function systems
as a finite collection of maps on a compact metric space
such that the diameter of the space by any combination of the maps goes to zero. More precisely, an iterated function system IFS$(\mathcal{F})$ is \emph{weakly hyperbolic} \cite{AJ, E} if for each $\omega=(\omega_1 \omega_2 \ldots \omega_n \ldots) \in \Omega^+$, $\textnormal{diam}(f_{\omega_1} \circ f_{\omega_2}\circ \ldots \circ f_{\omega_n}(X))\to 0$, whenever $n \to \infty$.


The iterated function system IFS$(\mathcal{F})$ has
\emph{covering property} if there exists an open set $D$ such that
$$D\subset \bigcup_{i=1}^k f_i(D).$$
Let $\Omega^-=\{1, \ldots, k\}^{\mathbb{Z}^-}$. Also let the fiber maps $f_i$, $i=1, \ldots, k,$ of IFS$(\mathcal{F})$ are uniformly  contracting. Therefore, $\textnormal{diam}(f_{\omega_{-1}}\circ \ldots \circ f_{\omega_{-n}}(X))$ tends to zero whenever $n \to +\infty$.
We define the \emph{limit set} of IFS$(\mathcal{F})$ by
\begin{equation*}\label{e41}
\Lambda:=\{x \in X: \exists \omega=(\ldots \omega_{-n}, \ldots, \omega_{-1}) \in \Omega^-, \ x=\lim_{n \to - \infty}f_{\omega_{-1}}\circ \ldots \circ f_{\omega_{-n}}(X)\}.
\end{equation*}
One easily proves that \cite{BCP} if IFS$(\mathcal{F})$ satisfies the covering property with the open set D, then the
limit set of the IFS contains $C^1$-robustly $D$. Hence the covering property is a sufficient condition for an IFS (with uniformly contracting fiber maps) to have $C^1$-robustly
 non-empty interior. Here, we improve this result to IFSs with nonuniform contracting fiber maps.
\begin{lemma}\cite[Corolary.~2.5]{AJ}\label{lem444}
Let IFS$(\mathcal{F})$ be a weakly hyperbolic finitely generated iterated function system with generators $\{f_1, f_2, \ldots, f_k \}$ on a compact metric space $X$. Then the limit
$$\Gamma(\omega, x)=\lim_{n \to +\infty} f_{\omega_1} \circ f_{\omega_2}\circ \ldots \circ f_{\omega_n}(x) $$
exists for every $\omega \in \Omega^+$ and $x \in X$, does not depend on $x$ and is uniform on $\omega$ and $x$.
\end{lemma}
Observe that the above lemma defines a map $\Gamma: \Omega^+ \to X$, given by
\begin{equation*}
\Gamma(\omega):= \lim_{n \to +\infty} f_{\omega_1} \circ f_{\omega_2}\circ \ldots \circ f_{\omega_n}(x), \ \textnormal{for \ any} \ x \in X.
\end{equation*}
By \cite[Lemma~2.7]{AJ}, the mapping $\Gamma$ is continuous in the product topology on $\Omega^+$. Moreover, IFS$(\mathcal{F})$ admits a unique
attractor $K$ with the basin $B(K)=X$. In particular, the attractor $K$ is defined \cite[Theorem~A]{AJ} as follows:
\begin{equation*}\label{e044}
  K:=\Gamma(\Omega^+)=\{\lim_{n \to +\infty} f_{\omega_1} \circ f_{\omega_2}\circ \ldots \circ f_{\omega_n}(x):\omega \in \Omega^+ \}.
\end{equation*}
\begin{remark}
The approach used in the proof of \cite[Thm.~A]{AJ} ensures that $K$ is a point-fibred attractor.
\end{remark}
\begin{lemma}
Let IFS$(\mathcal{F})$ be a weakly hyperbolic finitely generated iterated function system with generators $\{f_1, f_2, \ldots, f_k \}$ and a unique global attractor $K$. Also let IFS$(\mathcal{F})$ satisfy the covering property
\begin{equation}\label{e33}
  \overline{B} \subset f_1(B) \cup \ldots \cup f_k(B).
\end{equation}
Then $\overline{B} \subset K$. In particular, $K$ has nonempty interior.
\end{lemma}
\begin{proof}
First, we show that for each $x \in B$, there exists a sequence $(\omega_n)_{n \geq 1}$ of the alphabets $\{1, 2, \ldots, k\}$ so that
\begin{equation*}\label{34}
  x=\lim_{n \to + \infty}f_{\omega_1}\circ \ldots \circ f_{\omega_n}(y), \textnormal{for \ all} \ y \in B.
\end{equation*}
Indeed, we define the sequence $(\omega_n)_{n \geq 1}$ inductively: assume that we have found $\omega_1, \ldots, \omega_n \in \{1, 2, \ldots, k\}$ so that $x \in f_{\omega_1}\circ \ldots \circ f_{\omega_n}(B)$. Then the covering property (\ref{e33}) implies that
$$ x \in f_{\omega_1}\circ \ldots \circ f_{\omega_n}(B) \subset \bigcup_{i=1}^k f_{\omega_1}\circ \ldots \circ f_{\omega_n} \circ f_i(B)$$
and therefore we can find $\omega_{n+1}$ such that $x \in f_{\omega_1}\circ \ldots \circ f_{\omega_n}\circ f_{\omega_{n+1}}(B).$ Thus, we have constructed a sequence $\omega= (\omega_1, \omega_2, \ldots, \omega_n, \ldots) \in \Omega^+$ so that $$x \in \bigcap_{n \geq 1} f_{\omega_1}\circ \ldots \circ f_{\omega_n}(B).$$

Since IFS$(\mathcal{F})$ is weakly hyperbolic, $\textnormal{diam}(f_{\omega_1}\circ \ldots \circ f_{\omega_n}(B))$ tends to zero, whenever $n \to \infty$. This fact with together Lemma \ref{lem444} ensure that
$$x=\lim_{n \to + \infty}f_{\omega_1}\circ \ldots \circ f_{\omega_n}(B)=\lim_{n \to + \infty}f_{\omega_1}\circ \ldots \circ f_{\omega_n}(X), $$
and hence $x \in K$ by (\ref{e044}).
 In particular, by compactness of $K$, $\overline{B} \subset \overline{K}=K$.
\end{proof}
\subsection{\bf{First Step: Construction of a Weak Contractive Map and Cusp-like Region}}
In this section, we will discuss weak contractive maps defined on $\mathbb{R}^2$. For a certain class of these maps we
provide a bounded distortion property on a cusp-like region of the plane. First, we establish notations and provide some background information.

Consider a map $f : X \longrightarrow X$, where $X$ is a metric space.
We say that $f$ is \emph{weak contractive} (or \emph{non-expansive} \cite{E}) whenever for each $x,y \in X$, with $x\neq y$, $d(f(x),f(y)) < d(x,y)$.
It is a well-known fact \cite[Coro.~3]{J} (see also \cite{Bi}) that if $f$ is weak contractive on a compact metric space $X$ then there exists a unique fixed point
$x \in X$ of the map $f$. Furthermore, for every $y \in X$, $\lim_{k \to \infty}f^k(y)=x$ uniformly.
Then we say that $x$ is a \emph{weak attracting fixed point}.

Clearly if $f$ is a weak contractive map then
$$d(f^{n}(y), f^{n}(z)) \to 0, \ as, \ n \to \infty,$$
for each $y, z \in X$.


In the following, we assume that $f$ is a weak contractive $C^2$-diffeomorphism on $\mathbb{R}^2$.
Let $X$ be a closed ball of $\mathbb{R}^2$ so that $f(X) \subset X$. So we will focus on the dynamic of $f$ on the compact ball $X$. Clearly, $f$ is a weak contractive map on $X$.
\begin{definition}\label{def21}
We say that $x$ is a $\ast$-weak attracting fixed point of $f$ whenever

 $(1)$ $f$ is weakly contractive with weak attracting fixed point $x \in X$;

 $(2)$ $Df(x)$ has eigenvalues with different moduli.

$(3)$ $Df(x)$ possesses $1$ as an eigenvalue
and the other eigenvalue has modulus less than $1$.
\end{definition}
Now we fix any weak contractive $C^2$-diffeomorphism on a compact ball $X$ that admits a $\ast$-weak attracting fixed point $x \in X$.
Clearly $x$ is a non-hyperbolic fixed point of $f$.

In below, we illustrate the dynamical properties of $f$ at a neighborhood of $x$.
We will use the approach
proposed in \cite{ABV} to provide a bounded distortion property for the iterates of $f$
on a cusp-like region which is contained in a neighborhood of the point $x$.

According to assumptions $(2)$ and $(3)$ of the definition, $Df(x)$ admits a splitting $T_{x}(M)=E^s \oplus E^c$ with the following properties: there exists $0<\lambda<1$ that satisfies
\begin{equation}\label{e1}
\|Df(x)|_{E^s}\| \leq \lambda, \ \ \|Df(x)|_{E^s}\|. \|Df^{-1}(x)|_{E^c}\| \leq \lambda.
\end{equation}
We extend the subbundles $E^s$ and $E^c$ continuously to some neighborhood $V \subset X$ of $x$ and we denote by $\mathcal{E}^s$ and $\mathcal{E}^c$. Here, we do not require these extensions to be invariant under $Df$. For each $0<\beta<1$,  the \emph{center cone field} $\mathcal{C}_\beta ^c:= (\mathcal{C}_\beta ^c(y))_{y \in V}$ of width $\beta$ efined by
\begin{equation*}
\mathcal{C}_\beta ^c(y)=\{v_1 +v_2 \in \mathcal{E}^s_y \oplus \mathcal{E}^c_y : \|v_1\| \leq \beta \|v_2\|\}.
\end{equation*}
The \emph{stable cone field} $\mathcal{C}_\beta ^s:= (\mathcal{C}_\beta ^s(y))_{y \in V}$ of width $\beta$ is defined in a similar way. Fix $\beta>0$ and $V$ small enough so that up to increasing $\lambda<1$, the second inequality of (\ref{e1}) remains valid for any pair of vectors in the two cone fields:
\begin{equation*}\label{e3}
\|Df(y)v^s\|. \|Df^{-1}(f(y))v^c\| \leq \lambda \|v^s\|.\|v^c\|,
\end{equation*}
for every $v^s \in \mathcal{C}_\beta ^s(y), \ v^c \in \mathcal{C}_\beta ^c(y)$, and each point $y \in V \cap f^{-1}(V)$.
Then, the center cone field is positively invariant, that is
$Df(y)\mathcal{C}_\beta ^c(y)\subset \mathcal{C}_\beta ^c(f(y))$, provided that $y$ and $f(y)$ contained in $V$.
Indeed, according to (\ref{e1})
$$Df(x)\mathcal{C}_\beta ^c(x)\subset \mathcal{C}_{\lambda \beta} ^c(x) \subset \mathcal{C}_\beta ^c(x),$$
and this extends to each $y \in V \cap f^{-1}(V)$, by continuity.

We recall the notion of H\"{o}lder variation of the tangent bundle in local coordinates, as follows. Let $I \subset V$ be a $C^2$ embedded arc of $X$ which is tangent to the center cone field $\mathcal{C}_\beta ^c$ and take the exponential map on the embedded submanifold $I$. Suppose that $r>0$ is small enough so that the inverse of the exponential map $exp_x$ is defined on the $B_{r}(x,I):=B_{r}(x) \cap I$, where $B_{r}(x)$ is the ball with the radius $r$ and the center $x$ in $V$. We identify the neighborhood $B_{r}(x,I)$ of $x$ in $I$ with the corresponding neighborhood $U_x$ of the origin in $T_xI$, through the local chart defined by $exp^{-1}_x$. Then $T_yI$ is parallel to the graph of a unique
linear map
$$\mathcal{L}_x(y) : T_xI \to \mathcal{E}_x^s,$$
for more details see Section 2.1 of \cite{ABV}.
For given constants $C> 0$ and $0< \alpha \leq 1$, we say that the tangent bundle of $I$ is $(C, \alpha)$-{\it H\"{o}lder} if
\begin{equation*}
\|\mathcal{L}_x(y)\| \leq C \rho(x, y)^{\alpha} \, \textnormal{for} \ \textnormal{every} \ y \in I \cap U_x, \ x \in V,
\label{e1
}
\end{equation*}
where$\rho(x, y)$ denotes the distance from $x$ to $y$ along $I \cap U_x$.

The next result provide a bounded distortion property for iterates of $f$ over a cusp-like region that is contained in $V$.
 \begin{theorem}\label{thmcusp}
Let $f$ be a weak contractive map on $X$ with the $\ast$-weak attracting fixed point $x_0$.
Then there exist an open neighborhood $V$ containing $x_0$ a cusp-like region $W \subset V$ for which the following holds:

$(i)$ $W=\bigcup_{k=1}^\infty W_k,$ where $W_1$ is a closed rectangle in $V$ and $f(W_{k-1})=W_k$,
 for each $k \in \mathbb{N}$;

 $(ii)$ the functions
\begin{equation*}
J_{k} : W_k \ni y \mapsto \log |\textnormal{det}(Df|_{T_y(W_k)})|,
\end{equation*}
are $(L, \alpha)$-H\"{o}lder continuous, for some constant $L > 0$ and $\alpha =1$. In particular, $diam(W_k)$ tends to zero with a uniform rate whenever $k \to + \infty$.
\end{theorem}
\begin{proof}
First suppose that $I \subset V$ is a $C^2$ embedded arc of $X$ which is tangent to the center cone field $\mathcal{C}_\beta ^c$,
this means that the tangent subspace to $I$ at each point $x \in I$ is contained in the cone $\mathcal{C}_\beta ^c(x)$.
Then $f(I)$ is also tangent to the center cone field, since it is contained in $V$.
 We claim that the tangent bundle of the iterates of the $C^2$-submanifold $I$,
i.e. $f^n(I)$, $n \in \mathbb{N}$, are H\"{o}lder continuous
with uniform H\"{o}lder constant. Indeed by domination property (\ref{e1}) and the choice of $V$, there exist $\gamma \in (\lambda, 1)$
such that
\begin{equation}\label{e5}
\|Df(z)v^s\|. \|Df^{-1}(f(z))v^c\|^{1+ \alpha} \leq \gamma < 1,
\end{equation}
for each unit vectors $v^s \in \mathcal{C}_{\beta}^s(z)$, $v^c \in \mathcal{C}_{\beta}^c(z)$ and $z \in V$.
Now, by reducing $r$ and increasing $\gamma < 1$, (\ref{e5}) remains true if we replace $z$ by any $y \in U_x$, $x \in V$.

The curvature of $I$ is defined by
$$\kappa(I):=\inf \{C> 0: \textnormal{the \ tangent \ bundle \ of} \ I \ \textnormal{is} \ (C, \alpha)-\textnormal{Holder}\}.$$
Note that since $x_0$ is an $\ast$-weak attracting fixed point of $f$, $f^n(I)\subset V$, for all $n \geq 1$. So by Proposition 2.2 and Corollary 2.4 of \cite{ABV},
there exists $C_1 > 0$ for which the following statements hold:
\begin{itemize}
\item [(a)] there is an integer $n_0 \geq 1$ such that $\kappa (f^n(I)) \leq C_1$, for each $n \geq n_0$;
\item [(b)] if $\kappa(I) \leq C_1$, then $\kappa (f^n(I)) \leq C_1$;
\item [(c)] if $\kappa(I) \leq C_1$ then for each $k$ the function
\begin{equation*}\label{e6}
J_k : f^k(I) \ni y \mapsto \log |\textnormal{det}(Df|_{T_y(f^k(I))})|,
\end{equation*}
is $(L, \varepsilon)$-H\"{o}lder continuous with $L > 0$ depends only on $C_1$ and $f$.
\end{itemize}
Now, we consider the following form of domination
\begin{equation*}\label{e7}
\|Df(x)|_{E^s}\|. \|Df^{-1}(x)|_{E^c}\|^i \leq \lambda,
\end{equation*}
for $i=1, 2$. Since $f$ is $C^2$ then we may take $\alpha =1$ in the above argument.
In particular, since $I$ is a $C^2$ curve and by $(b)$, the curvature of all iterates $f^n(I)$, $n \geq 1$, is bounded by some constant that depends only on the curvature of $I$. We assume that the origin is the $\ast$-weak attracting fixed point of $f$ and
the one dimensional center bundle $E^c$ coincide with the $x_2$-axis. Take an arc $D$ so that it is orthogonal to the $x_2$-axis. Also, we assume that the distance between the arc $D$ and the origin is sufficiently small.
We fix a constant $C_1$ for which the statements $(a)$, $(b)$ and $(c)$ hold and then
we consider a family of $C^2$-embedded arcs $\{ I_x : x \in D \}$ such that each $I_x$
intersects the arc $D$ transversally at the point $x$ and connect the point $x$ to $f(x)$.
Moreover, by taking $D$ small enough, we may choose these arcs so that
for each $x \in D$, $I_x\subset V$, it is tangent to the center cone field $\mathcal{C}_\beta ^c$ with
$\kappa(I_x) \leq C_1$, and also they vary continuously with the base point $x$ in the $C^2$ topology, see figure \ref{fi:1}.
Since $f$ is $C^2$ and by Statement $(c)$, for each $x \in D$, the functions
\begin{equation*}
J_{k, x} : f^k(I_x) \ni y \mapsto \log |\textnormal{det}(Df|_{T_y(f^k(I_x))})|,
\end{equation*}
are $(L, \varepsilon)$-H\"{o}lder continuous with $L > 0$ depends only on $C_1$ and $f$ and $\varepsilon =1$. Now, we take a closed rectangle $W_1\subset \bigcup _{x \in D}I_x$ so that its vertical edges are $I_{x_1}$ and $I_{x_2}$, for some points $x_1, x_2 \in D$, and we define $W_k$, $k \in \mathbb{N}$, inductively by $W_k=f(W_{k-1})$. Also we take
$W= \bigcup _{k=1}^\infty W_k$. It is not hard to see that $W$ and $W_k$,  $k \in \mathbb{N}$, satisfying the conclusion of the theorem.
\begin{figure}[h]
\begin{center}
\includegraphics[scale=1.3]{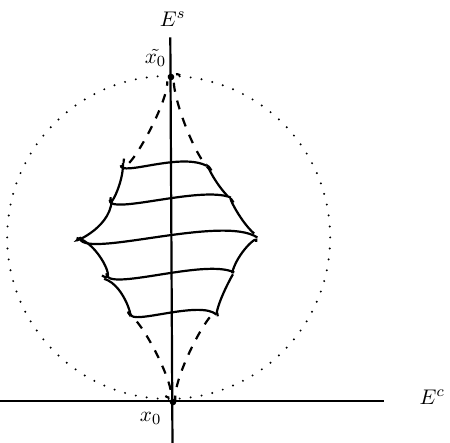}
\caption{Cusp-like Region}\label{fig:1}
\end{center}
\end{figure}
\end{proof}
\subsection{Second Step: a Non Hyperbolic IFS with a Non Degenerate Attractor}
As in the first step, let $X$ be an open ball in $\mathbb{R}^2$. We recall that the \emph{Lyapunov exponents} of an IFS$(\mathcal{F})$ on $X$ are numbers that occur as limits
$$\lim_{n \to \infty} \frac{1}{n}\textnormal{log}\|Df_\omega^n(x)v\|$$
for $x \in X$, nonzero vectors $v \in T_x(X)$ and $\omega \in \Omega^+$. For an ergodic stationary measure $\mu$, one has that
$\mathbb{P}^+ \times \mu$ almost everywhere there are Lyapunov exponents $\lambda_1 \geq \lambda_2$ , not
depending on $x$ or $\omega$. Note that the \emph{top Lyapunov exponent} $\lambda_1$
is computed as
\begin{equation}\label{e22}
\lambda_1=\lim_{n \to \infty} \frac{1}{n}\textnormal{log}\|Df_\omega^n(x)\|.
\end{equation}
Now we state the main result of this section.
\begin{theorem}\label{thm1}
There exists a weakly hyperbolic (non-hyperbolic) iterated function system $\textnormal{IFS}(\mathcal{F})$ generated by a finite number of weakly contractive diffeomorphisms on $X$ that admits a unique global attractor $K$ for which the following holds:
\begin{itemize}
  \item [(1)] $K$ has nonempty interior;
  \item [(2)] $\textnormal{IFS}(\mathcal{F})$ admits an ergodic measure $\mu$ with $\textnormal{supp}(\mu)=K$;
  \item [(3)] the top Lyapunov exponent of $\textnormal{IFS}(\mathcal{F})$ is negative.
\end{itemize}
\end{theorem}
The rest of this subsection is devoted to prove of the above theorem. Without loss of generality, we may assume that $X$ is a compact ball of $\mathbb{R}^2$ containing the origin and $f:X \to X$, introduced in he first step, is a weak contractive map for which the point $x_0$ is an $\ast$-weak attracting fixed point. Therefore, $f(x_0)=x_0$ and $Df(x_0)$ possesses 1 as an eigenvalue and the other eigenvalue $\lambda$
satisfies $0<\lambda<1$. Consider the arc $D$ applied in the proof of Theorem \ref{thmcusp} which is orthogonal to the central subbundle  $E^c$ ($x_2$-axis).
 Let $z$ be the intersection point of $D$ with the central subbundle  $E^c$ and $W$ be the cusp-like region of $X$ provided by Theorem \ref{thmcusp}. Suppose that $\alpha$ is the distance between the points $z$ and $x_0$. We take $S$ the circle with the radius $\alpha$ and the center $z$.
From now on, by a smooth change of coordinates, we may assume that the circle $S$ is the standard unit circle $S^1$.
Assume that $R_\beta$ is an irrational rotation defined by $R_\beta(r,\theta)=(r,\theta+\beta)$
in polar coordinate and $T:= R_{\pi/2}$. Take $\widetilde{f}:= T \circ f \circ T^{-1}$ and assume that $\widetilde{W}:= \widetilde{f}(W)$ is the mirror image of the cusp-like region $W$ by the reflecting map $\widetilde{f}$. Put $A:=W \cup \widetilde{W}$.  Suppose that $B$ is an ellipse-like region
so that the major axis lies on the central bundle $E^c$ and
$\overline{B} \subset A$. Also, suppose that $B$ satisfies $$int(B)\cap (A \setminus(f^{-1}(W_1)\cup W)\neq \emptyset \
\textnormal{and} \ int(B)\cap (A \setminus (\widetilde{f}^{-1}(\widetilde{W}_1)\cup \widetilde{W})\neq \emptyset,$$ see figure \ref{fig:2}, below.
\begin{figure}[h]
\begin{center}
\includegraphics[scale=1.3]{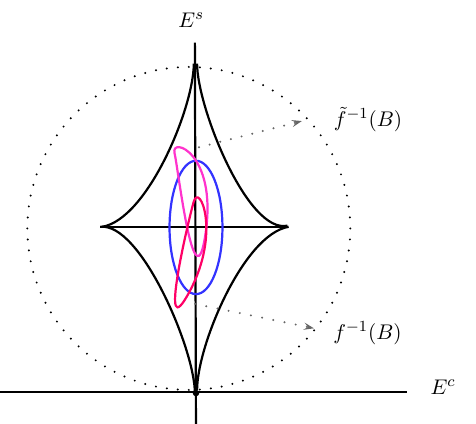}
\caption{Cusp-like Region}\label{fig:2}
\end{center}
\end{figure}
Now,  the choice of $B$ and the bounded distortion property  provided by Theorem \ref{thmcusp} ensure that $f(B) \cup \widetilde{f}(B)$ is connected and
the diameter of $f(B) \cup \widetilde{f}(B)$ along the $E^c$-direction is greater than the length of the major axis of $B$.
 Putting $\widetilde{B}:=f(B) \cup \widetilde{f}(B)$, it is not hard to see that there exist positive integers $n_1, \ldots, n_m $
in such a way that for $B_i:= R_\beta^{n_i}(\widetilde{B})$, $i=1, \ldots, m$, one has
\begin{equation}\label{e233}
\overline{B} \subset B_1 \cup \ldots \cup B_m.
\end{equation}
Now, take a hyperbolic contracting $C^2$ diffeomorphism $g$ defined on $X$ so that it possesses a unique attracting fixed point at the point $x_0$ and satisfies the following properties:
\begin{itemize}
  \item $Dg(x_0)$ has two real eigenvalues $\lambda_1$ and $\lambda_2$, with $0<\lambda_1<\lambda_2<1$, such that $\lambda_2$ is close enough to 1 and $\lambda_1=\lambda$, where $\lambda$ is the smallest eigenvalue of $Df(x_0)$.
  \item $g=f$ outside an arbitrary small neighborhood of $x_0$.
 \end{itemize}
 We take
\begin{equation}\label{e44}
f_1:=f, \ f_i:= R_\beta^{n_i} \circ g, \ \textnormal{for \ each} \ i=2, \ldots, m, \ f_{m+1}:=\widetilde{f} \ \textnormal{and} \
f_{i+m}:= T \circ f_i, \ i=2, \ldots, m.
\end{equation}
We show that the iterated function system IFS($\mathcal{F}$) generated by the maps $\{f_1, \ldots, f_{2m}\}$ defined in (\ref{e44})
satisfies the conclusion of the Theorem \ref{thm1}.
\begin{lemma}\label{lem22}
Let IFS($\mathcal{F}$) be the iterated function system generated by the maps $\{f_1, \ldots, f_{2m}\}$ defined in $\textnormal{(\ref{e44})}$. Then IFS($\mathcal{F}$) is weakly hyperbolic.
\end{lemma}
\begin{proof}
Since $f_1=f$ and $f_{m+1}=\widetilde{f}$ are weak contractive maps, by the iterates of these two maps, the entire of $X$ collapses  to the fixed points $x_0$ and $\widetilde{x}_0$ of $f$ and $\widetilde{f}$, respectively. Note that the point $\widetilde{x}_0$ is the mirror image of $x_0$. Moreover, the central bundles of these two maps are the same. These facts imply that the iterated function system $\mathcal{F}(X; f, \widetilde{f})$ is a weakly hyperbolic IFS according to Proposition 2.5 of \cite{E}.

To prove the lemma, let $\omega=(\omega_1 \omega_2 \ldots \omega_n \ldots)\in \Omega^+$. If the set
$\{n \in \mathbb{N}: \omega_n \neq 1, m+1\}$ is infinite then clearly $\lim_{n \to \infty}\textnormal{diam}(f_{\omega_1} \circ \ldots \circ f_{\omega_n}(X))=0$,
since the mappings $f_i$, $i\neq 1, m+1$, are uniformly contracting on $X$.
Otherwise, if the above set is finite then it
has a maximum element $l$. Hence for all $n > l$, $\omega_n \in \{1, m+1\}$ and
therefore $$\textnormal{diam}(f_{\omega_1} \circ \ldots \circ f_{\omega_l}\circ \ldots \circ f_{\omega_n}(X))\leq \textnormal{diam}(f_{\omega_{l+1}} \circ \ldots \circ f_{\omega_n}(X)).$$
But $\textnormal{diam}(f_{\omega_{l+1}} \circ \ldots \circ f_{\omega_n}(X))$ tends to zero,
since for each $n > l$, $\omega_n \in \{1, m+1\}$ and the iterated function system $\mathcal{F}(X; f, \widetilde{f})$ is weakly hyperbolic. This fact
 implies that
 $$\textnormal{diam}(f_{\omega_1} \circ \ldots \circ f_{\omega_n}(X)) \to 0, \ \textnormal{as} \ n \to \infty.$$ Therefore, IFS($\mathcal{F}$) is weakly hyperbolic as we have claimed.
 \end{proof}
Since the generators $f_i$, $i=1, \ldots, 2m$, are weakly contractive, by \cite[Thm.~2.6]{Bi}, IFS$(\mathcal{F})$ possesses a unique global attractor $K$ (see also \cite[Thm.~A]{AJ}).

Let us note that the attractor of each weakly hyperbolic iterated function system satisfies the point fibred property, this means that $K$ is a point fibred global attractor. Moreover, the inclusion (\ref{e233}) implies the following covering property
\begin{equation}\label{e24}
 \overline{B} \subset f_1(B) \cup \ldots \cup f_{2m}(B).
\end{equation}
In view of the covering property, Lemma \ref{lem22} ensures that $\overline{B} \subset K$. In particular, the attractor $K$ has nonempty interior. On the other hand, the mappings $f_i$ are weak contractive and for $i \neq 1, m+1$, $f_i$ is uniformly contracting which ensure that IFS$(\mathcal{F})$ satisfies the contractive on average property. Now, Proposition \ref{mainprop}, implies the existence of an ergodic stationary measure $\mu$ whose support is $K$.

It remains to show that the top fiber Lyapunov exponent of IFS$(\mathcal{F})$ is negative.
Since $\mu$ is an ergodic stationary measure, one has that
$\mathbb{P}^+ \times \mu$ almost everywhere the top Lyapunov exponent $\lambda_1$ defined by the limit (\ref{e22}), not
depending on $x$ or $\omega$. So it suffices to prove that for $\mathbb{P}^+$-almost surely $\omega$ and all $x$, the limit (\ref{e22}) is negative. The attractor $K$ is a point-fibred attractor, hence for each $\omega$, we can associate a point $x(\omega) \in K$ for which
the sequence $f_{\omega_1}\circ \ldots \circ f_{\omega_n}(X)$ tends to it whenever $n \to \infty$. Since, $K$ has nonempty interior, by \cite{BLV}, $x(\omega) \in int(K)$, provided that $\omega$ is disjunctive.  On the other hand, by \cite[Section.~4]{BV},
the set of all disjunctive sequences forms a full measure subset of $\Omega^+$. This means that for $\mathbb{P}^+$-almost surely $\omega$,
the corresponding point $x(\omega)$ is contained in the interior of $K$. This fact and contractive on average property of IFS$(\mathcal{F})$ imply that for $\mathbb{P}^+$-almost surely $\omega$ and all $x$, the top Lyapunov exponent (\ref{e22}) is negative.
\subsection{Thirth Step: Constructing the Random Iterated Function Systems}
This subsection is devoted to study random iterated function systems acting on $S^1 \times X$ of the form
 \begin{equation*}
\Psi :S^1 \times X \to X, \  \Psi(t,x)=f_{t}(x),
\end{equation*}
where, as before, $X \subset \mathbb{R}^2$ is a closed topological ball containing the origin and the fiber maps $x\mapsto f_{t}$ are
$C^2$-diffeomorphisms defined on $X$. Moreover, the \emph{parameter space} is $S^1$. We say that $\Psi$ is a \emph{random iterated function system}.
Let $f_i$, $i=1, \ldots, 2m$, be the diffeomorphisms so that the
iterated function system IFS($\mathcal{F})$ generated by these
maps satisfies the conclusion of Theorem \ref{thm1} and let
$\varphi:S^1\to S^1$ be an expanding map
given by
$$\varphi(x)=kx~~(mod~1),$$
where $k=4m$. For irrational number $\beta$ chosen in the second
step, take points $t_1=0$ and $t_i=n_i \beta$ (mod 1), $i=2,
\ldots, m$, $t_{m+1}=\frac{1}{2}$ and $t_{m+i}=t_i+ \frac{1}{2}$
(mod 1), $i=2, \ldots, m$, of the unit interval [0,1]. By
renumerating the indices, we may assume that
\begin{equation}\label{e14}
0=t_1 < t_2 < \ldots < t_m < t_{m+1}= 1/2< t_{m+2} < \ldots < t_{2m}<t_{2m+1}=1.
\end{equation}
Let $\delta > 0$ be small enough
  and take a smooth function $\eta : [0,1] \to [0,\delta]$ satisfying
\begin{equation*}
\eta(t) :=\begin{cases}
\begin{array}{ccc}
\eta_1(t) & \text{for \ t $\in$ [0,$\delta $]},\\
\eta_2(t) & \text{ \ \ \ \ for \ t $\in$ [$\delta,\frac{1}{2}-\delta$]},\\
     \eta_3(t) & \text{ \ \ \ \ for \ t $\in$ [$\frac{1}{2}-\delta,\frac{1}{2}$]},\\
     \eta_4(t) & \text{ \ \ \ \ \ for \ t $\in$ [$\frac{1}{2},\frac{1}{2}+\delta$]},\\
      \eta_5(t) & \text{ \ \ \ \ \ \ \ \ for \ t $\in$ [$\frac{1}{2}+\delta,1- \delta$]},\\
       \eta_6(t) & \text{ \ \ \ \ for \ t $\in$ [$1- \delta,1$]},
\end{array}
  \end{cases}
  \label{e4}
  \end{equation*}
where $\eta_1$ is an increasing smooth function with $\eta_1(0)=0$, $\eta_2(t):=\delta$, $\eta_3(t)$ is a decreasing smooth function with $\eta_3(1/2)=0$, $\eta_4(t)$ is a increasing smooth function with $\eta_4(1/2)=0$,
$\eta_5(t):=\delta$ and $\eta_6(t)$ is a decreasing smooth function with $\eta_6(1)=0$. Now, assume that $\lambda^{\prime}$ is the largest eigenvalue of $Dg(0)$ and take $\delta=1- \lambda^{\prime}$.
Moreover, the points $t_i$ defined by (\ref{e14}) satisfy
\begin{equation}\label{e15}
 0 =t_1 < \delta <t_2 < \ldots < t_{m}< 1/2-\delta < t_{m+1}=1/2<1/2+\delta < t_{m+2} < \ldots < t_{2m}< t_{2m+1}=1.
\end{equation}
Assume that
$$A(t)= \left( {\begin{array}{*{20}{c}}
1-t&0\\
0&\lambda
\end{array}} \right),$$
where $0<\lambda < 1$ is the smallest eigenvalue of $Df(0)$. For each $t \in [0,1]$, let $R_t$ be the rotation on $S^1$ with the angle $2\pi t$. Now, we define a smooth arc of $C^2$ diffeomorphims on $X$ of the form
\begin{equation*}
 H: [0,1] \to \textnormal{Diff}^2(X), \ t \mapsto h_t
\end{equation*}
such that
\begin{itemize}
  \item [(1)] $\forall t \in [0,1], \ h_t$ is $C^2$ close to the identity and admits a unique fixed point at $x_t$;
  \item [(2)] $Dh_t(x_t)= R_t \circ A(\eta(t))$;
  \item [(3)] $h_{t}(X) \subset \textnormal{int}(X)$;
  \item [(4)] $\forall t\in [0,1] \ \forall x \neq x_t, \|Dh_t(x) \| < 1$;
  \item [(5)] $h_{t_i}=f_i$, $i=1, \ldots, 2m$ and $h_1=f_1$.
\end{itemize}
Let $I_i \subset S^1$, $i=1, \ldots, 2m$, be disjoint closed arcs
around $t_i$ with length $1/k$. Then $\varphi(I_i)=S^1$, $i=1,
\ldots, 2m$. Consider a smooth map $\theta: S^1 \to [0,1]$ such that for each
$i=1, \ldots, 2m +1$, $\theta |_{ I_i}=t_i$ and put
$f_t:=h_{\theta(t)}$, for each $t$. Now, we define the iterated
function system $\Psi$ by
\begin{equation*}\label{2.2.1}
 \Psi:S^1 \times X \to X, \ (t,x)\mapsto f_t(x).
\end{equation*}
The choice of $f_t$ ensures that $\Psi$ is continuous. Take
$\Sigma^+=(S^1)^{\mathbb{N}}$ and equip it with the product
topology.
\begin{lemma}
The iterated function system $\Psi$ is weakly hyperbolic.
\end{lemma}
\begin{proof}
Let $\textbf{t}=(t_1 t_2 \ldots t_n \ldots) \in \Sigma^+$. Since
$\theta |_{ I_i}=t_i$, so for each $t \in I_i$,
$f_t:=h_{\theta(t)}=f_i$. Put $\mathcal{A}=\{n \in \mathbb{N}: t_n
\in \cup_{i=1}^{2m}I_i\}$. We consider two cases
\begin{itemize}
\item {\bf $\mathcal{A}$ is infinite}.~~Then there exists a
subsequence $t_{i_j}, \ j\geq 1$, of $\textbf{t}=(t_1 t_2 \ldots
t_n \ldots)$ so that $t_{i_j} \in \cup_{i=1}^{2m}I_i$ and hence
$f_{t_{i_j}}\in \{f_1, \ldots, f_{2m}\}$. This fact implies that
 $$\lim_{n \to \infty}\textnormal{diam}(f_{t_1} \circ \ldots \circ f_{t_n}(X))\leq \lim_{n \to \infty}\textnormal{diam}(f_{t_{i_1}} \circ \ldots \circ f_{t_{i_n}}(X))=0,$$
since the iterated function system generated by the mappings $f_i$, $i=1, \ldots, 2m$, is weakly hyperbolic, according to Lemma \ref{lem22}.\\
\item  {\bf $\mathcal{A}$ is finite}.~~$\mathcal{A}$ has at most
$\ell$ elements. Therefore,
$$\textnormal{diam}(f_{t_1} \circ \ldots \circ f_{t_\ell}\circ \ldots
\circ f_{t_n}(X))\leq \textnormal{diam}(f_{t_{\ell+1}} \circ \ldots
\circ f_{t_n}(X)).$$ But $\textnormal{diam}(f_{t_{\ell+1}} \circ
\ldots \circ f_{t_n}(X))$ tends to zero, since for each $n > \ell$, by part (2) of (\ref{e15}) and the definition of the fiber maps $f_t$,
the mapping $f_{t_n}$ is uniformly contracting. This fact
 implies that
 $$\textnormal{diam}(f_{t_1} \circ \ldots \circ f_{t_n}(X)) \to 0, \ \textnormal{as} \ n \to \infty.$$ Therefore, $\Psi$ is weakly hyperbolic as we have claimed.
\end{itemize}
\end{proof}
\begin{remark}
We can define the Hutchinson operator for the random iterated function system $\Psi$ by $\mathcal{L}(A)=\bigcup_{t \in S^1}f_t(A)$, for each compact subset $A \subset X$. In a similar way, we say that the random iterated function $\Psi$ admits a global attractor $\Lambda$ if
$\mathcal{L}^n(B) \to \Lambda$ whenever $n \to \infty$, for each compact set $B$ of $X$. Moreover, $\Lambda$ is invariant set by $\Psi$ and it is the unique fixed point of the operator $\mathcal{L}$. In particular, by Theorem A of \cite{AJ} the attractor $\Lambda$ is defined by
\begin{equation*}
  \Lambda:=\{\lim_{n \to +\infty} f_{t_1} \circ f_{t_2}\circ \ldots \circ f_{t_n}(X):\textbf{t}=(t_1, \ldots, t_n, \ldots) \in \Sigma^+ \}.
\end{equation*}
\end{remark}
Now, we consider the global attractor $\Lambda$ of $\Psi$ which is obtained from the previous remark.
By the construction of $\Psi$, the condition (5) of (\ref{e15}) and the covering property (\ref{e24}) in Theorem \ref{thm1}, there exists a topological ball $B \subset X$ so that
\begin{equation}\label{e17}
  \textnormal{Cl}(B) \subset \bigcup_{i=1}^{2m} f_{t_i}(B) \subset \bigcup_{t \in S^1}f_t(B).
  \end{equation}
By applying the argument used in Lemma \ref{lem22}, the next lemma follows by any modifications.
\begin{lemma}\label{lemmainterior}
The unique global attractor $\Lambda$ of the iterated function system $\Psi$ has nonempty interior.
Moreover, $ \textnormal{Cl}(B) \subset \Lambda$.
\end{lemma}
Let $\varrho:\Sigma^+ \to \Sigma^+$ be the left shift on
$\Sigma^+=(S^1)^{\mathbb{N}}$. For each sequence $\textbf{t}=(t_1
t_2 \ldots t_n \ldots) \in \Sigma^+$, we denote
$\textbf{t}|_n:=(t_1 t_2 \ldots t_n )$ and
$f_{\textbf{t}|_n}:=f_{t_1} \circ \ldots \circ f_{t_n}$. We define
the \emph{coding map}
\begin{equation}\label{e18}
 \tau: \Sigma^+ \to X, \ \tau(\textbf{t})= \lim f_{\textbf{t}|_n}(X).
\end{equation}
Since $\Psi$ is weakly hyperbolic, the mapping $\tau$ is
well-defined and continuous by \cite[Lemma.~2.7]{AJ}. Now, we
define the associated skew product of $\Psi$ by
\begin{equation}\label{e16}
  F_\Psi:S^1 \times X \to S^1 \times X, \ F_\Psi(t,x)=(\varphi(t),\Psi(t,x))=(\varphi(t),f_t(x)),
\end{equation}
where, as before, $\varphi(t)=kt \ (\textnormal{mod} \ 1)$. In the
rest, we show that there exists an open neighborhood of $F_\Psi$ in the space of $\mathcal{C}(\mathbb{T})$ which satisfies
the conclusion of the main theorems.
\section{Skew products}
In this section we will discuss skew products of the form
\begin{equation*}
F:Y \times X \to Y \times X, \ F(y,x)=(\theta(y),f(y,x)),
\end{equation*}
where $X$ is a compact ball of $\mathbb{R}^2$ and $Y$ be a metric space equipped with a Borel measure $\nu$ and suppose that $\theta : Y \to Y$
is a continuous map ergodic with respect to $\nu$.
\subsection{Skew products and Their Extensions}
Let us consider the associated skew product $F_\Psi$ of the IFS $\Psi$ over the expanding circle map $\varphi$ that is defined by (\ref{e16}).
Note that for iterates of $F_\Psi$, we denote
$$F_\Psi^{n}(t,x)=(\varphi^{n}(t),f_{\varphi^{n-1}(t)} \circ \ldots \circ f_{t}(x))=(\varphi^{n}(t),f^{n}_{t}(x)).$$
Clearly the expanding map $\varphi$ is not invertible. We mention that if $\varphi:S^1 \rightarrow S^1$ is a $C^2$-expanding endomorphism then
$\varphi$ possesses an absolutely continuous invariant ergodic measure $\nu^+$ whose density is bounded and bounded away from zero (see \cite{M}).
Now we define the natural extension of $(\varphi, S^1, \nu^+)$. Indeed,
by the inverse limit construction (\cite{Wi}), $\varphi$ admits an extension $\xi$ which is a homeomorphism on the solenoid, i.e. the space
$$\mathcal{S}=\{(\ldots,t_{-1},t_{0})\in (S^1)^{\mathbb{Z}^-}: t_{-j}=\varphi(t_{-j-1})\}$$
endowed with the product topology, defined by
\begin{equation*}\label{3.9}
 \xi(\textbf{t})=(\ldots,t_{-1},t_{0},\varphi(t_0)), \ \textnormal{for} \ \textbf{t}=(\ldots,t_{-1},t_{0}).
\end{equation*}
 The induced skew product map on $\mathcal{S}\times X$, denoted by $G_\Psi$, has the form
\begin{equation}\label{e66}
 G_\Psi(\textbf{t},x)=(\xi(\textbf{t}),g(\textbf{t},x))=(\xi(\textbf{t}),f_{\textbf{t}}(x)),
\end{equation}
where $f_{\textbf{t}}(x):=f_{t_0}(x)$
and the inverse map is given by
\begin{equation*}
G_\Psi^{-1}(\textbf{t},x)=(\ldots,t_{-2},t_{-1},(f_{t_{-1}})^{-1}(x)).
\end{equation*}
Consider the projection
\begin{equation}\label{e99}
\pi:\mathcal{S}\rightarrow S^1, \pi(\textbf{t})=t_0.
\end{equation}
Then $\pi \circ \xi=\varphi \circ \pi$ which means that $\pi$ is a semiconjugacy between $F_\Psi$ and $G_\Psi$.
by Kolmogorov's extension theorem, the solenoid $\mathcal{S}$ has an invariant measure $\nu$ inherited from the invariant measure $\nu^+$ for $\varphi$ on the circle. 
Now,  consider an invariant measure $\mu^+$ for $F_\Psi$ on $S^1\times X$ with marginal $\nu^+$ whose existence is guaranteed by \cite{Cr} and let
$\mu_{t}^+$ be its disintegrations. The following lemma clarifies relation between invariant measures for $F_\Psi$ and its extension $G_\Psi$.
\begin{lemma}[\cite{Cr, H}]
Given the invariant measure $\mu^+$ for $F_\Psi$ acting on $S^1 \times X$, with marginal $\nu^+$ on $S^1$, there is an invariant measure $\mu$ for $G_\Psi$ acting on $\mathcal{S}\times X,$ with marginal $\nu$ on $\mathcal{S}$. For $\nu$-almost all $\textnormal{\textbf{t}}=(\ldots,t_{-1},t_{0})\in \mathcal{S},$ the limit
$$\mu_{\textnormal{\textbf{t}}}=\lim_{n\rightarrow \infty}f^{n}_{t_{-n}}\mu_{t_{-n}}^+$$
gives its disintegrations.
Conversely, given an invariant measure $\mu$ for $G_\Psi$ on $\mathcal{S}\times X$,
$$\mu_t^+ =\mathbb{E}(\mu \mid \mathcal{F}^+)_t$$
is an invariant measure for $F_\Psi$ on $S^1\times X.$ Moreover, the correspondence maps ergodic measures to ergodic measure in either direction.
\label{L1.1}
\end{lemma}
\subsection{Maximal Attractors and Continuous Invariant Graphs}
In this subsection, we prove the existence of a non-hyperbolic maximal attractor for the skew product $G_\Psi$ and we show that it
is a continuous invariant graph.
Indeed, since the region $\mathcal{S} \times X$ is trapping by $G_\Psi$ (see part (3) of (\ref{e15})), it admits a maximal attractor
\begin{equation*}
  \Delta_G:=\bigcap_{n\geq 0}G_\Psi^n(\mathcal{S} \times X).
\end{equation*}
It is enough to show that $\Delta_G$ is a continuous invariant graph, that is there exists a continuous map $\gamma:\mathcal{S} \to X$ whose graph, $\Gamma$, coincides with $\Delta_G$ and satisfies
$$G_\Psi(\textbf{t}, \gamma(\textbf{t}))=(\xi(\textbf{t}),\gamma(\xi (\textbf{t})) ),$$ for each $\textbf{t} \in \mathcal{S}$.
\begin{proposition}\label{pro1}
The maximal attractor of $G_\Psi$ is a continuous invariant graph.
\end{proposition}
The rest subsection is devoted to prove the proposition. For each $x \in X$ and $\textbf{t} \in \mathcal{S}$ let us take
\begin{equation*}
 X(\textbf{t},n,x):= f_{\textbf{t}} \circ f_{\xi^{-1}(\textbf{t})}\circ \dots \circ f_{\xi^{-n}(\textbf{t})}(x),
\end{equation*}
\begin{lemma}
For each $x \in X$ and $\textbf{t} \in \mathcal{S}$ the sequence $X(\textbf{t},n,x)$ is convergent and the limit does not depend on $x$.
\end{lemma}
\begin{proof}
Since the iterated function system $\Psi$ is weakly hyperbolic the following holds (\cite{AJ}): for each $\epsilon >0$ there exists $n_{0}=n_{0}(\epsilon) \in \mathbb{N}$ such that for all $n\geq n_{0}$ and any $\textbf{t} \in \mathcal{S}$ we have that $\textnormal{diam} (f_{\textbf{t}} \circ f_{\xi^{-1}(\textbf{t})}\circ \dots \circ f_{\xi^{-n}(\textbf{t})}(X))< \epsilon$. Let $\epsilon >0$ be given. For each $x \in X$ and $\textbf{t} \in \mathcal{S}$
 we observe that  $X (\textbf{t},n,x)\in f_{\textbf{t}} \circ f_{\xi^{-1}(\textbf{t})}\circ \dots \circ f_{\xi^{-n}(\textbf{t})}(X) $ and
 $$X(\textbf{t},n+\ell,x)\in f_{\textbf{t}} \circ f_{\xi^{-1}(\textbf{t})}\circ \dots \circ f_{\xi^{-n-\ell}(\textbf{t})}(X)\subset f_{\textbf{t}} \circ f_{\xi^{-1}(\textbf{t})}\circ \dots \circ f_{\xi^{-n}(\textbf{t})}(X).$$ Thus one gets
 $$d(X(\textbf{t},n+\ell,x),X(\textbf{t},n,x))<\epsilon,$$
 for all $n\geq n_{0}$ and $\ell\in \mathbb{N}$. This shows that the sequence
 $X(\textbf{t},n,x) $ is Cauchy and thus convergent for all $x\in X$ and $\textbf{t} \in \mathcal{S}$. Since $n_{0}$ does not depend on $\textbf{t}$, the limit $\lim_{n \to \infty} f_{\textbf{t}} \circ f_{\xi^{-1}(\textbf{t})}\circ \dots \circ f_{\xi^{-n}(\textbf{t})}(x)$ is uniform on $\textbf{t}$ and $x$.
 Now, for each $x\in X$ and $\textbf{t} \in \mathcal{S}$,
 $$X(\textbf{t},n,x) , X(\textbf{t},n,y)\in f_{\textbf{t}} \circ f_{\xi^{-1}(\textbf{t})}\circ \dots \circ f_{\xi^{-n}(\textbf{t})}(X)$$
 and then
 $$\lim_{n \to +\infty} d (X(\textbf{t},n,x) , X(\textbf{t},n,y))=0$$
 which shows that the limit does not depend on $x$.
\end{proof}
The above lemma motivates us to define the map $\gamma$ by
\begin{equation*}\label{e1111}
\gamma: \mathcal{S} \to X, \ \gamma(\textbf{t}):= \lim_{n \to + \infty}f_{\textbf{t}} \circ f_{\xi^{-1}(\textbf{t})}\circ \dots \circ f_{\xi^{-n}(\textbf{t})}(0).
\end{equation*}
Let us consider the following classical metric on $\mathcal{S}$.  For each $\textbf{t}=(\ldots, t_{-1}, t_0), \textbf{t}^{\prime}=(\ldots, t^{\prime}_{-1}, t^{\prime}_0)\in \mathcal{S}$,
$$d_S(\textbf{t},\textbf{t}^{\prime}):=\sum_{i=0}^\infty \frac{d(t_{-i},t_{-i}^{\prime})}{2^i},$$
\begin{lemma}\label{lem34}
The map $\gamma$ is continuous.
\end{lemma}
\begin{proof}
Fix $\textbf{t} \in \mathcal{S}$ and $\epsilon >0$. By lemma above, there exists $m=m(\epsilon)$ such that  for all $\textbf{t}$ and $x$,
$$d(f_{\textbf{t}} \circ f_{\xi^{-1}(\textbf{t})}\circ \dots \circ f_{\xi^{-m}(\textbf{t})}(x), \gamma (\textbf{t}))<\epsilon$$
By the continuity of $f_{\textbf{t}} \circ f_{\xi^{-1}(\textbf{t})}\circ \dots \circ f_{\xi^{-m}(\textbf{t})}$, there is $\delta >0$ such that $d(f_{\textbf{t}} \circ f_{\xi^{-1}(\textbf{t})}\circ \dots \circ f_{\xi^{-m}(\textbf{t})}(x),f_{\textbf{t}^{\prime}} \circ f_{\xi^{-1}(\textbf{t}^{\prime})}\circ \dots \circ f_{\xi^{-m}(\textbf{t}^{\prime})}(x)< \varepsilon$, for all $x$, provided $d_S(\textbf{t},\textbf{t}^{\prime})< \delta$ then. Let $U$ be the neighborhood of $\textbf{t}$ given by
$$U=\ldots \times S^1 \times B(\xi^{-m}(\textbf{t}),\delta)\times B(\xi^{-1}(\textbf{t}),\delta) \times B(\textbf{t},\delta). $$
For $\textbf{t}^{\prime}\in U$ one gets
\begin{eqnarray*}
d(\gamma(\textbf{t}),\gamma(\textbf{t}^{\prime}))&\leq&
d(\gamma(\textbf{t}),f_{\textbf{t}} \circ f_{\xi^{-1}(\textbf{t})}\circ \dots \circ f_{\xi^{-m}(\textbf{t})}(x))\\
&+&d(f_{\textbf{t}} \circ f_{\xi^{-1}(\textbf{t})}\circ \dots \circ f_{\xi^{-m}(\textbf{t})}(x),f_{\textbf{t}^{\prime}} \circ f_{\xi^{-1}(\textbf{t}^{\prime})}\circ \dots \circ f_{\xi^{-m}(\textbf{t}^{\prime})}(x))\\&+&d(\gamma(\textbf{t}^{\prime}),f_{\textbf{t}^{\prime}} \circ f_{\xi^{-1}(\textbf{t}^{\prime})}\circ \dots \circ f_{\xi^{-m}(\textbf{t}^{\prime})}(x))\\&<&3\varepsilon.
\end{eqnarray*}
This shows that $\gamma$ is continuous.
\end{proof}
Let us consider the maximal attractor $\Delta_G$ of $G_\Psi$ and take $X_\textbf{t}=\{\textbf{t}\}\times X$. Then
$$ \Delta_\textbf{t}:=\Delta_G \bigcap X_\textbf{t} = \bigcap_{n\geq 0} X(\textbf{t},n),$$
It is not hard to see that $\Gamma$, the graph of $\gamma$, is equal to $\Delta_G$. On the other hand by the definition, the equality
$G_\Psi(\textbf{t}, X(\textbf{t},n,0))=(\xi(\textbf{t}),X(\xi(\textbf{t}),n+1,0))$
holds. This says that
$$G_\Psi(\textbf{t}, \gamma(\textbf{t}))=(\xi(\textbf{t}),\gamma(\xi (\textbf{t})) ).$$
This means that $\Gamma$ is an invariant graph. Now, in view of Lemma \ref{lem34} the proof of the proposition follows.
\subsection{Negative Maximal Lyapunov Exponents}
In the following, we concern the ergodic properties of the skew product $F_\Psi(t,x)=(\varphi(t), \Psi(t,x))$
and its extension $G_\Psi$.

Note that in our setting $\Psi(t,.) \in \textnormal{Diff}^2(X)$. We define a norm $\|.\|$ by
\begin{equation*}
 \| \Psi(t,.)\|:=\sup_{x \in X} \|D \Psi(t,x)\|
\end{equation*}
and we consider the Lipschitz metric
\begin{equation*}
  \rho(\Psi(t_1,.),\Psi(t_2,.)):=\|\Psi(t_1,.)-\Psi(t_2,.) \|+\sup_{x \in X}|\Psi(t_1,x)-\Psi(t_2,x)|.
\end{equation*}
Since the fiber maps $\Psi(t,.) \in \textnormal{Diff}^2(X)$ and $\nu^+$ is a probability measure, a compactness argument ensure that $\textnormal{log}^+\| \Psi(t,.)\| \in L^1(t,\nu^+)$. By Oseledec's classical theorem (\cite[Thm.~5.4]{Kr}), the limit
\begin{equation*}
 \lambda(t)=\lim_{n \to \infty}\frac{1}{n}\textnormal{log} \| \Psi^n(t,.)\|
\end{equation*}
exists for $\nu^+$ a.e. $t \in S^1$, where $\Psi^n(t,x)$ defined implicitly by $F_\Psi^n(t,x)=(\varphi^n(t),\Psi^n(t,x))$. By ergodicity of $(\varphi, S^1, \nu^+)$,
there exists $\lambda$ such that $\lambda(t)=\lambda$ for $\nu^+$ a.e. $t \in S^1$. We call $\lambda$ the \emph{maximal fiber Lyapunov exponent} of $F_\Psi$ (with respect to $\nu^+$). In fact, putting $\eta_n(t):=\log \| \Psi(\varphi^n(t),.)\|$, one gets the subassitive property $\eta_{n+l}(t) \leq \eta_n(t) + \eta_l(t)$,  \textnormal{for} \ \textnormal{all} $n,l > 0$. According to \cite[Thm.~5.4]{Kr},
\begin{equation*}
 \lim_{n \to \infty}\frac{1}{n}\textnormal{log} \| \Psi^n(t,.)\|=\inf \frac{1}{n}\textnormal{log} \| \Psi^n(t,.)\|.
\end{equation*}
Consider a perturbed skew product $\widetilde{F} \in \mathcal{C}(\mathbb{T})$ of $F_\Psi$ defined by $ \widetilde{F}(t,x)=(\varphi(t), \widetilde{\Psi}(t,x)$.
If
\begin{equation*}
 \lim_{n \to \infty}\frac{1}{n}\textnormal{log} \| \Psi^n(t,.)\|=\lambda < 0, \ \textnormal{for} \ \nu^+ \ a.e. \ t \in S^1,
\end{equation*}
for some $\lambda  < 0$ then given a sufficiently small $\varepsilon > 0$ there exists $\delta > 0$ such that if
$\rho(\Psi(t,.),\widetilde{\Psi}(t,.))< \delta$, for $\nu^+$ a.e. $t \in S^1$, then
$$ \lim_{n \to \infty}\frac{1}{n}\textnormal{log} \| \widetilde{\Psi}^n(t,.)\|=\lambda + \epsilon < 0, \ \textnormal{for} \ \nu^+ \ a.e. \ t \in S^1.$$
Consequently, given $\varepsilon > 0$ there exists a measurable function $C : S^1 \to \mathbb{R}^+$ such that for $\nu^+$ a.e. $t \in S^1$,
\begin{equation*}\label{3.7}
 \| \widetilde{\Psi}^n(t,.)\| < C(t)e^{(\lambda + \varepsilon)n}, \ \textnormal{for} \ \textnormal{all} \ n> 0.
\end{equation*}
The next lemma points out that the skew product $F_\Psi$ and its extension $G_\Psi$ possessing the negative maximal fiber Lyapunov exponents which ensure that, in particular, the maximal fiber Lyapunov exponent of $\widetilde{F}$ is also negative.
\begin{lemma}\label{neg-exponents}
The skew product map $F_\Psi$ and its extension $G_\Psi$ defined by (\ref{e66}) admit negative maximal fiber Lyapunov exponents with respect to $\nu^+$ and $\nu$, respectively.
\end{lemma}
\begin{proof}
To prove the lemma first we show that the skew product system $F_\Psi$ admits negative maximal fiber Lyapunov exponent with respect to $\nu^+$.
The proof of the lemma is based on some well known facts of number theory.
A number $t \in [0,1]$ is called a $k$-rich if its expansion $\omega(t,k)$ in base $k$ is disjunctive.
It is known that $t$ is $k$-rich if and only if 
the set $\{k^nt \ (\textnormal{ mod } \ 1): n \in \mathbb{N} \}$ is dense
in the unit interval \cite{BY}, this means that the orbit of $t$ under
the mapping $\varphi$ is dense. By  Normal Number Theorem (\cite{FS}),
Lebesgue almost all real numbers are $k$-rich.
Let $t \in [0,1]$ with digit sequence expansion $\theta=\omega(t,k)$.
We claim that  $t$ is a rich number if and only if $\tau(\theta) \in \textnormal{int}(\Lambda)$,
where $\tau$ is the coding map defined by (\ref{e18}) and $\Lambda$ is the unique global attractor
of IFS $\Psi$. Note that by Lemma \ref{lemmainterior}, $\Lambda$ has nonempty interior.

Indeed, since $t$ is a rich number, its digit sequence expansion $\theta$ is disjunctive and
therefore the orbit of $\theta$ under the shift map $\varrho$ is dense in $\Sigma^+$.
By the continuity of $\tau$, the orbit of $\tau(\theta)$ is also dense in $\Lambda$. Now, by the way of contradiction,
assume that $\tau(\theta) \in \partial \Lambda$, the boundary of $\Lambda$. Since each $f_t$ is a homeomorphism, $f_t(\textnormal{int}(\Lambda)) \subset \textnormal{int}(\Lambda)$ or, equivalently,
$(f_t | _{\Lambda})^{-1}(\partial \Lambda) \subset \partial \Lambda$. This implies that the orbit of $\tau(\theta)$ completely lies in $\partial \Lambda$, which would mean that $\Lambda=\overline{\mathcal{O}(\tau(\theta), \Psi)} \subset \overline{\partial \Lambda}=\partial \Lambda$. This lead to the contradiction  $\textnormal{int}(\Lambda)=\emptyset$.

Notice that for all $t \in S^1$, the fiber map $f_t$ is contacting in $\textnormal{int}(\Lambda)$
These observations imply that for $\nu^+$-almost all points $t$,
the maximal fiber Lyapunov exponent of $t$ is negative.
Moreover, it is constant by ergodicity, we denote it by $\lambda$.

Now we prove that the maximal fiber Lyapunov exponent of $G_\Psi$ is also negative.
In fact, as a consequence of the modified Multiplicative Ergodic Theorem
(see \cite[Prop.~2.2]{BHN}), there exists $\varepsilon > 0$ (which is obtained by semi-conjugacy $\pi$) so that for $\nu$-almost every $\textbf{t} \in \mathcal{S}$ there exists a constant
$C(\textbf{t})$ such that
$$\|Df_{\xi^n(\textbf{t})}\ldots Df_{\textbf{t}}\|<C(\textbf{t})e^{(\lambda+\varepsilon)n},$$
$$\|Df_{\textbf{t}}\ldots Df_{\xi^{-n}(\textbf{t})}\|<C(\textbf{t})e^{(\lambda+\varepsilon)n},$$
for all $n \geq 0$, see also \cite[Section.~2]{BHN}.
This fact terminates the proof.
  \end{proof}
\subsection{Ergodic Properties of the Skew Products}
In this subsection we study the ergodic properties of $G_\Psi$. To start, we need some preliminary definition.  Let $(X; \mathcal{B}; \nu; T)$ be a measure preserving dynamical system. If $T$ is invertible then, the system is  {\it Bernoulli} if it is isomorphic to a Bernoulli shift. For  non-invertible cace, being Bernoulli means that the natural extension to the inverse limit space is isomorphic to a one-sided Bernoulli shift. $T$ is mixing (or strong mixing) if
$$\nu(T^{-n}(A) \cap B)\to \nu(A)\nu(B), \ as \ n \to +\infty,$$
for every $A,B \in \mathcal{B}$.

The definition of a SRB measure \cite{Pa} is expressed in rather simple terms, but its content is quite meaningful in describing
the dynamics of attractors when it's possible to show that they carry SRB-measures.

Let $A$ be an attractor for $f$, i.e. there is a set of points in the phase space with positive
Lebesgue probability whose future orbits tend to $A$, as the number of iterates tends to infinite.
The set of orbits attracted to $A$ in the future is called its basin and denoted by $B(A)$. Let $\mu$ be an $f$-invariant probability measure on $A$.
Following in \cite{Pa}, $\mu$ is called a \emph{SRB (Sinai-Ruelle-Bowen) measure} for $(f,A)$ if we have for any continuous map $g$, that
$$\lim_{n \to \infty}\frac{1}{n}\sum_{i=0}^{n-1}g(f^i(x))=\int g d \mu, $$
for all $x \in E \subset B(A)$, with $m(E) > 0$, where $m$ denotes Lebesgue measure.
It's common to have $E$ with total probability in $B(A)$ and so, in such a case, the above
convergence holds for a typical trajectory attracted to $A$ in the future.

One of the most important results which we shall use is the following theorem.
\begin{theorem}(\cite[Thm.~5]{C2}, \cite[Thm.~1.4]{St2} and \cite[Pro.~2.3]{BHN})\label{thm000}
Let $X$ be a compact ball of $\mathbb{R}^n$, $S$ a compact metric space equipped with a Borel measure $\nu$ and $h:S \to S$ is a continuous map ergodic with respect to $\nu$. Also let $F:S \times X \to S \times X$, $F(t,x)=(h(t),f(t,x))$ be a skew product on $S \times X$ and for each $t\in S$, $f(t,.)\in \textnormal{Diff}^2(X)$. Assume $h$ is invertible and the skew product $F$ admits a negative fiber Lyapunov exponent. Then there is a measurable function $\gamma:S \to X$ such that $\Gamma$ the graph of $\gamma$ is invariant under $F$. The graph $\Gamma$ (which may not be compact) is the support of an invariant measure $\mu$ and $(F,\Gamma,\mu)$ is ergodic (strong mixing, Bernoulli) if $(h,S,\nu)$ is ergodic (strong mixing, Bernoulli). Furthermore $\Gamma$ is attracting in the sense that for $\nu \ a.e. \ t\in S$, $\lim_{n \to \infty}|\pi_x(F^n(t,x))-\pi_x \Gamma(t)|=0$ for every $x \in X$, where $\pi_x$ is the natural projection from $S \times X$ to $X$.
\end{theorem}

Now, we summarize some basic ergodic properties of the map $G_\Psi$ in the following theorem.
\begin{theorem}\label{thm22}
Take the skew product $G_\Psi$ defined by $\textnormal{(\ref{e66})}$ with the maximal attractor $\Delta_G$.
Then the following statements hold:
\begin{enumerate}
 \item There exist a continuous function $\gamma:\mathcal{S} \to X$ such that $\Gamma$ the graph of $\gamma$ coincides with the unique maximal attractor $\Delta_G$ of $G_\Psi$. Moreover, $\Gamma$ is invariant under $G_\Psi$, that is for each $\textbf{t} \in \mathcal{S}$, $G_\Psi(\textbf{t},\gamma(\textbf{t}))=(\xi(\textbf{t}), \gamma(\xi(\textbf{t})))$.
  \item $G_\Psi$ admits a negative maximal Lyapunov exponent.
    \item The graph $\Gamma$ is the support of an ergodic invariant measure $\mu$.
  \item $(G_\Psi, \Gamma, \mu)$ is Bernoulli and therefore it is mixing.
  \item $\mu$ is an SRB measure.
\end{enumerate}
  \end{theorem}
\begin{proof}
The statement $(1)$ follows by Proposition \ref{pro1}. The second statement follows from Lemma \ref{neg-exponents}. The base map $\xi$ is invertible hence we can apply Theorem \ref{thm000} to conclude that
the graph $\Gamma$ is the support of an invariant ergodic measure $\mu$ (for more details see \cite{BHN}), hence the statement $(3)$ is satisfied. To verify the statement $(4)$, we observe that the map $\varphi, \ \varphi(t)=kt \ (\textnormal{mod \ 1})$, is a $C^2$-expanding map of the circle preserving the Lebesgue measure. Therefore its natural extension $\xi$ is Bernoulli (see \cite{W, Q}). Now, by Theorem \ref{thm000}, $(G_\Psi, \Gamma, \mu)$ is also Bernoulli and so mixing.

It remains to show that $\mu$ is an SRB measure. To this end, we show that for all $\textbf{t}\in \mathcal{S}$ and all $x\in X$
and any continuous function $h \in C(\mathcal{S}\times X)$,
\begin{equation}\label{e55}
  \lim_{k \to +\infty} \frac{1}{k}\sum_{i=0}^{k-1}h(G_\Psi^i(\textbf{t},x))=\int h d\mu.
\end{equation}
By Lemma \ref{neg-exponents}
$$\textnormal{dist}(G_\Psi^k(\textbf{t},x), G_\Psi^k(\textbf{t},\gamma(\textbf{t})))\to 0, $$
as $k \to +\infty$, uniformly in $\textbf{t}$ and in $x$. By the continuity of $h$,
$$h(G_\Psi^k(\textbf{t},x))- h(G_\Psi^k(\textbf{t},\gamma(\textbf{t})))\to 0.$$
Therefore to prove (\ref{e55}), it is enough to prove it for $x=\gamma(\textbf{t})$, i.e.
\begin{equation}\label{e77}
  \lim_{k \to +\infty} \frac{1}{k}\sum_{i=0}^{k-1}h(G_\Psi^i(\textbf{t},\gamma(\textbf{t})))=\int h d\mu.
\end{equation}
Since the projection $p : \Delta_G \to \mathcal{S}$, $(\textbf{t},x)\mapsto \textbf{t}$ is an isomorphism, the function
$\widetilde{h}=h \circ (p|_{\Delta_G})^{-1}$ is continuous on $\mathcal{S}$. Therefore, (\ref{e77}) is equivalent to
\begin{equation*}\label{e78}
  \lim_{k \to +\infty} \frac{1}{k}\sum_{i=0}^{k-1}\widetilde{h}(\xi^i(\textbf{t})=\int_{\mathcal{S}} \widetilde{h} d \nu
\end{equation*}
for almost all $\textbf{t}$. This statement is just the ergodicity of $\nu$.
\end{proof}
Consider the projection $\mu_{F}=(\pi \times id)_{\ast}\mu$, where
$\pi: \mathcal{S} \to S^1$ is defined by (\ref{e99}). Since $\mu$ is an invariant ergodic measure
 with $\textnormal{supp}(\mu)=\Gamma$ and by \cite[Pro.~3.5]{VY}, $\mu_F$ is also an invariant ergodic measure for $F_\Psi$. Also, by the construction, clearly $(\pi \times id)\Gamma$ is a global attractor for $F_\Psi$ with the basin $S^1 \times X$. These observation and the argument used in the statement $(5)$ of the previous theorem ensure that $\mu_F$ is a SRB measure of $F_\Psi$.
%
\begin{proposition}
The following holds:
\begin{equation*}\label{e101}
S^1 \times \textnormal{Cl}(B) \subset (\pi \times id)\Gamma.
\end{equation*}
\end{proposition}
\begin{proof}
Take an open set $U \subset X$ so that $$\textnormal{Cl}(B) \subset \bigcup_{i=1}^{2m} f_{i}(B) \subset \bigcup_{t \in S^1} f_{t}(B) \subset U $$
and each $f_{t}$ is contracting on $\textnormal{Cl}(U)$, the first inclusion given by (\ref{e17}).
Notice that for each $t \in S^1 \setminus \cup_{i=1}^{2m}I_i$, the fiber maps $f_t$ are uniformly contracting which ensure that for all $t \in S^1$ the map $f_t$ is contracting on $\textnormal{Cl}(U)$. Therefore we can take a uniform upper rate of contraction of the fiber maps in $\textnormal{Cl}(U)$
$$ \lambda = \sup_{t \in S^1, \ x \in \textnormal{Cl}(U)}\|Df_t(x)\|.$$
For each $t \in S^1$, we denote the fibers $X_t$, $U_t$ and $B_t$
$$X_t:=\{t\}\times X, \ U_t:=\{t\}\times U \ \textnormal{and} \ B_t:= \{t\}\times \textnormal{Cl}(B),$$
and we take $\Gamma_t:=X_t \cap (\pi \times id)\Gamma$. By Theorem \ref{thm22}, there exists a continuous function $\gamma : \mathcal{S} \to X$ such that $\Gamma$ the graph of $\gamma$ coincides with the maximal attractor $\Delta_G$ of $G_\Psi$ and it is invariant under $G_\Psi$.

Now, we claim that for any $t \in S$ and $n \in \mathbb{N}$, there exists a finite covering of $B_t$
by balls of radius $\lambda^{n-1} \varepsilon$ whose centers
lie in the set $\Gamma_t$, where $\varepsilon=\textnormal{diam}(B)$. As $n$ is arbitrary, the claim proves the density of $\Gamma_t$ in $B_t$ for any $t \in S$ and thus the density of
$(\pi \times id)\Gamma$ in $S^1 \times \textnormal{Cl}(B)$. Finally, the fact
that $(\pi \times id)\Delta_G$ has to be a closed set finishes the proof of the Proposition.

The proof of the claim is analogous to \cite[Pro.~7.1]{V} and it will be carried out by induction.
Indeed, for the case $n=1$ we observe that any point from the non-empty set $\Gamma_t$ suffices as the center of a single
ball of radius $\textnormal{diam}(B)$, which clearly contains $B_t$.

Now assume that for some $n \in \mathbb{N}$ and for any $t \in S^1$, there exists a finite covering of
$B_t$ by the balls of the radius $\lambda^{n-1}(\textnormal{diam}(B))$ with centers in the set $\Gamma_t$. Denote by $O_t$ the set
of all centers of the balls.

Take any $t^{\prime} \in S^1$. There exist $2m$ preimages of $t^{\prime}$, $t_i^{\prime}$, $i=1, \ldots, 2m,$ within $I_i$, respectively:
$$\varphi(t_i^{\prime})= t^{\prime}, \ t_i^{\prime}\in I_i, \ i=1, \ldots, 2m.$$
Because of the invariance of $\Delta_G$ and $(\pi \times id)\Delta_G$, we have
$$\Gamma_{t^{\prime}}=\bigcup_{t \in \varphi^{-1}(t^{\prime})}f_t(\Gamma_t)\supset \bigcup_{i=1}^{2m}f_{t_i^{\prime}}(\Gamma_{t_i^{\prime}}). $$
We know that every $f_t$ contracts $U_t$ with the uniform upper rate of $\lambda < 1$. This
observation, combined with the induction assumption, gives us that the balls of radius $\lambda^{n}(\textnormal{diam}(B))$ with centers in $f_t(O_t)$ constitute a covering of the regions $f_t(B_t) \subset U_t$, $t \in \varphi^{-1}(t^{\prime})$. Now note that the maps $f_{t_i^{\prime}}$ are equal to the maps $f_i$ originally defined by (\ref{e44}), and therefore
$$\textnormal{Cl}(B)\subset \bigcup_{i=1}^{2m}f_{t_i^{\prime}}(B).$$
Thus the balls of radius $\lambda^{n}(\textnormal{diam}(B))$ with the centers in $\cup_{i=1}^{2m}f_{t_i^{\prime}}(O_{t_i^{\prime}})$
cover the whole region $B$. So the induction step is complete.
\end{proof}
Notice that we can apply the argument used in \cite[Section~8]{V} to deduce the existence of an open neighborhood $U$ of $\textnormal{supp}(\mu_F)$ satisfies the following:
\begin{equation}\label{e444}
\textnormal{supp}(\mu_F)=\bigcap_{n\geq 0}F_\Psi^n(U).
\end{equation}
By this fact, the previous proposition and \cite[Corollary~2.4]{BHN}, analogous to Theorem \ref{thm22} we get the next result.
\begin{theorem}
 Take the skew product $F_\Psi$ defined by $\textnormal{(\ref{e16})}$. Then $F_\Psi$ admits an attractor $\Delta$ that satisfies the following properties:
\begin{enumerate}
  \item $\Delta$ has nonempty interior and it is the support of an ergodic invariant measure $\mu_F$.
  \item $\mu_F$ is an SRB measure.
  \item $(F_\Psi, \Delta, \mu_F)$ is Bernoulli and therefore it is mixing.
  \end{enumerate}
\end{theorem}

\section{Proof of the main results}
In this section, we assume that $\widetilde{F}$ is a small perturbation of $F_\Psi$ with respect to the metric defined by (\ref{e000}) of the following form
\begin{equation*}\label{e333}
\widetilde{F}:S^1 \times X \to S^1 \times X, \ \widetilde{F}(t,x)=(\varphi(t),\widetilde{f}(t,x))=(\varphi(t),\widetilde{f}_t(x))
\end{equation*}
and
 $\widetilde{G}$ is the extension of $\widetilde{F}$ that is defined by
 \begin{equation*}
  \widetilde{G}:\mathcal{S} \times X \to \mathcal{S} \times X, \ \widetilde{G}(\textbf{t}, x)=(\xi(\textbf{t}), \widetilde{g}(\textbf{t},x)).
 \end{equation*}
Here, we generalize the concept of a bony graph \cite{KV} to our setting. We say that a closed invariant set of the skew product $\widetilde{G}$ is a {\it bony graph}  if it intersects almost every fiber (w.r.t. $\nu$) at a single point, and any other fiber at a compact connected set which is called a bone. A bony graph can be represented as a disjoint union of two sets, $K$ and $\widetilde{\Gamma}$, where $K$ denotes the union of the
bones. The projection of $K$ by the natural projection map $p$ to the base has zero measure, while $\widetilde{\Gamma}$ is the graph of some
measurable function $\widetilde{\gamma} : \mathcal{S} \setminus p(K) \to X$.
By Fubini's Theorem, the standard measure of a bony graph is zero.
Let $\widetilde{\Delta}$ be the maximal attractor of $\widetilde{G}$. Take $X_\textbf{t}:=\{\textbf{t}\}\times X$. By definition
\begin{equation*}\label{e23}
\widetilde{\Delta}_\textbf{t}:=\widetilde{\Delta} \cap X_\textbf{t}=\bigcap_{n\geq 0} X(\textbf{t},n), \ \textnormal{where} \ X(\textbf{t},n)=\widetilde{f}_\textbf{t} \circ \widetilde{f}_{\xi^{-1}(\textbf{t})}\circ \ldots \circ \widetilde{f}_{\xi^{-n}(\textbf{t})}(X).
\end{equation*}
Let the maximal attractor $\widetilde{\Delta}$ be a bony graph. Then we say that $\widetilde{\Delta}$ is a \emph{continuous-bony graph} (CBG) if $\widetilde{\Delta}_\textbf{t}$ is upper semicontinuous:
\begin{equation*}\label{e25}
  \forall \textbf{t} \forall \varepsilon>0 \ \exists \delta >0 \ \ \textnormal{such \ that} \ \ \textnormal{dist}(\textbf{t}, \textbf{t}^{\prime})< \delta
  \ \Longrightarrow \widetilde{\Delta}_{\textbf{t}^{\prime}}\subset U_\varepsilon(\widetilde{\Delta}_\textbf{t}),
\end{equation*}
where $U_\varepsilon(\widetilde{\Delta}_\textbf{t}):=\bigcup_{a\in \widetilde{\Delta}_\textbf{t}} B_\varepsilon (a)$ and $B_\varepsilon(a)$ is a ball with center $a$ and radius $\varepsilon$. In other words, $U_\varepsilon(\widetilde{\Delta}_\textbf{t})$ is the $\varepsilon$-neighborhood of $\widetilde{\Delta}_\textbf{t}$ in the Hausdorff metric.
In particular, the graph part $\widetilde{\Gamma}$ is a graph of a function which is continuous on its
domain.

\begin{theorem}\label{thm44}
 Take the skew product $\widetilde{G}$ which is a small perturbation of $G_\Psi$ with the maximal attractor $\widetilde{\Delta}$.
Then the following statements hold:
\begin{enumerate}
  \item $\widetilde{\Delta}$ is a continuous bony graph. There exists an upper semicontinuous function $\widetilde{\gamma}:\mathcal{S} \to X$
  so that $\widetilde{\Gamma}$, the graph of $\widetilde{\gamma}$ is invariant under $\widetilde{G}$.
  \item $\widetilde{G}$ admits a negative maximal Lyapunov exponent.
    \item The graph $\widetilde{\Gamma}$ is the support of an ergodic invariant measure $\widetilde{\mu}$.
  \item $(\widetilde{G}, \widetilde{\Gamma}, \widetilde{\mu})$ is Bernoulli and therefore it is mixing.
  \item $\widetilde{\mu}$ is an SRB measure.
\end{enumerate}
  \end{theorem}
\begin{proof}
Notice that, by Lemma \ref{neg-exponents}, $F_\Psi(t,x)=(\varphi(t), \Psi(t,x))$ admits a negative maximal fiber Lyapunov exponent, say $\lambda$.

Consider a perturbed skew product $\widetilde{F} \in \mathcal{C}(\mathbb{T})$ of $F_\Psi$ defined by $ \widetilde{F}(t,x)=(\varphi(t), \widetilde{\Psi}(t,x)$.
This means that $\sup_{t} dist_{C^1}(f_t^{\pm}, \widetilde{f}_t^{\pm})< \delta$, for sufficiently small $\delta > 0$, where $\widetilde{f}_t, \ t \in S^1$, are the fiber maps of $\widetilde{F}$.
Since
\begin{equation*}
\lim_{n \to \infty}\frac{1}{n}\textnormal{log} \| \Psi^n(t,.)\|=\lambda < 0, \ \textnormal{for} \ \nu^+ \ a.e. \ t \in S^1,
\end{equation*}
then given a sufficiently small $\varepsilon > 0$ there exists $\delta > 0$ such that if
$\rho(\Psi(t,.),\widetilde{\Psi}(t,.))< \delta$, for $\nu^+$ a.e. $t \in S^1$, then
\begin{equation*}
\lim_{n \to \infty}\frac{1}{n}\textnormal{log} \| \widetilde{\Psi}^n(t,.)\|=\lambda + \epsilon < 0, \ \textnormal{for} \ \nu^+ \ a.e. \ t \in S^1.
\end{equation*}
Consequently, given $\varepsilon > 0$ there exists a measurable function $C : S^1 \to \mathbb{R}^+$ such that for $\nu^+$ $a.e.$ $t \in S^1$, we have
\begin{equation*}\label{3.7}
 \| \widetilde{\Psi}^n(t,.)\| < C(t)e^{(\lambda + \varepsilon)n}, \ \textnormal{for} \ \textnormal{all} \ n> 0.
\end{equation*}
In particular, small perturbations $\widetilde{F}$ and $\widetilde{G}$ of $F_\Psi$ and $G_\Psi$, respectively, possessing negative maximal Lyapunov exponents.

Now, we prove that the maximal attractor $\widetilde{\Delta}$ of the perturbed skew product $\widetilde{G}$ is a continuous bony graph. Note that since the base map $\xi$ is ergodic and invertible then $\xi^{-1}$ is also ergodic. Hence
$\xi^{-1}$ is ergodic with respect to $\nu$ and has the same spectrum of Lyapunov exponents by Furstenberg-Kesten Theorem \cite{FK}, see also \cite[Pro.~2]{E2}. We define
$$\psi_n(\textbf{t},x):=\widetilde{g}^n(\xi^{-n}(\textbf{t}),x).$$
The above fact 
ensure that
$$  \lim_{n \to \infty}\frac{1}{n}\textnormal{log} \| \psi_n(\textbf{t},.)\|=\lambda + \epsilon < 0, \ \textnormal{for} \ \nu \ a.e. \ \textbf{t} \in \mathcal{S},$$
where $\lambda < 0$ is the maximal fiber Lyapunov exponent of $G_\Psi$.
Hence there exists $\ell(\textbf{t})$ such that
$$\| \psi_n(\textbf{t},.)\|<e^{n(\lambda + \varepsilon)} \ \ \forall n\geq \ell(\textbf{t}).$$
Thus given $\varepsilon > 0$ there exists a measurable function $C : \mathcal{S} \to \mathbb{R}^+$ such that for $\nu$ a.e. $\textbf{t} \in \mathcal{S}$, we have
\begin{equation*}
 \| \psi_n(\textbf{t},.)\| < C(\textbf{t})e^{(\lambda + \varepsilon)n}, \ \textnormal{for} \ \textnormal{all} \ n> 0.
 \end{equation*}
Now we apply the approach used in the proof of \cite[Thm.~3.1]{BHN} to ensure that for fixed $x \in X$ and $a.e. \ \textbf{t} \in \mathcal{S}$, the sequence $\{\psi_\ell(\textbf{t},x)\}$ is a
Cauchy sequence. Let $$K(x)=\sup_{\textbf{t} \in \mathcal{S}}|\textbf{t}-\widetilde{g}(\textbf{t},x)|$$
and note that for fixed $x$, $K(x)$ is finite as $X$ is compact and $\widetilde{g}$ continuous. Given any $\varepsilon^{\prime}>0$, choose $\ell^{\ast}(\textbf{t})$ sufficiently large such that
$$K(x)c(\textbf{t})\sum_{j=\ell^{\ast}(\textbf{t})}^\infty e^{j(\lambda + \varepsilon)}< \varepsilon^{\prime}. $$
Now, for $m >\ell>\ell^{\ast}(\textbf{t})$ one gets
$$|\psi_m(\textbf{t},x)-\psi_\ell(\textbf{t},x)| \leq k(x)\sum_{j=\ell}^\infty \|\psi_j(\textbf{t},.) \| \leq k(x)c(\textbf{t})\sum_{j=\ell}^\infty e^{j(\lambda + \varepsilon)}< \varepsilon^{\prime}.$$
To see this note that
$$|\psi_m(\textbf{t},x)-\psi_\ell(\textbf{t},x)|=|\psi_m(\textbf{t},x)-\psi_{m-1}(\textbf{t},x)+ \ldots +\psi_{\ell +1}(\textbf{t},x)-\psi_\ell(\textbf{t},x)|.$$
Applying $\widetilde{G}$ to $(\xi^{-k}(\textbf{t}),x)$, one gets $\widetilde{G}(\xi^{-k}(\textbf{t}),x)=(\xi^{-(k-1)}(t),\widetilde{g}(\xi^{-k}(\textbf{t}),x))$.
Thus, $\psi_k(\textbf{t},x)-\psi_{k-1}(\textbf{t},x)=\psi_{k-1}(\textbf{t},\widetilde{g}(\xi^{-k}(\textbf{t}),x)))- \psi_{k-1}(\textbf{t},x)$.
As a result, $$|\psi_{k-1}(\textbf{t},x)-\psi_{k}(\textbf{t},x)| \leq \|\psi_{k-1}(\textbf{t},.) \| |x- \widetilde{g}(\xi^{-k}(\textbf{t}),x))|. $$ Hence
\begin{eqnarray*}
|\psi_m(\textbf{t},x)-\psi_\ell(\textbf{t},x)| &\leq& |\psi_m(\textbf{t},x)-\psi_{m-1}(\textbf{t},x)|+ \ldots +|\psi_{\ell +1}(\textbf{t},x)-\psi_\ell(\textbf{t},x)|\\
&=&\sum_{j=\ell+1}^m | \psi_j(\textbf{t},x)-\psi_{j-1}(\textbf{t},x)|\\
&\leq& \sum_{j=\ell+1}^m \| \psi_{j-1}(\textbf{t},.) \| |x-\widetilde{g}(\xi^{-j}(\textbf{t}),x))|\\
&\leq& \sum_{j=\ell+1}^\infty \| \psi_{j-1}(\textbf{t},.) \| K(x).
\end{eqnarray*}
Finally, for $\ell > \ell^{\ast}(\textbf{t})$ one has that
$$|\psi_m(\textbf{t},x)-\psi_\ell(\textbf{t},x)| \leq  K(x) \sum_{j=\ell+1}^\infty \| \psi_{j-1}(\textbf{t},.) \|
=K(x)c(\textbf{t})\sum_{j=\ell+1}^\infty e^{(j-1)(\lambda + \varepsilon)}< \varepsilon^{\prime}.$$
This shows that for every $x \in X$ and $a.e. \ \textbf{t} \in \mathcal{S}$ the sequence  $\{\psi_m(\textbf{t},x)\}$ is a
Cauchy sequence. Now, define $$\widetilde{\gamma}(\textbf{t})= \lim_{n \to + \infty}\psi_n(\textbf{t},0).$$
Since $$\widetilde{G}(\textbf{t}, \psi_\ell(\textbf{t},0))=(\xi(\textbf{t}),\psi_{\ell +1}(\xi(\textbf{t}),0)), $$
we see that
$$\widetilde{G}(\textbf{t}, \widetilde{\gamma}(\textbf{t}))=(\xi(\textbf{t}),\widetilde{\gamma}(\xi(\textbf{t})) ) $$
and hence $\widetilde{\Gamma}$, the graph of $\widetilde{\gamma}$ is invariant under $\widetilde{G}$.
Furthermore, by construction
$$\lim_{n \to +\infty}|\widetilde{g}^n(\textbf{t},x)-\widetilde{g}^n(\textbf{t},0)|=0. $$
This is because
$$|\widetilde{g}^n(\textbf{t},x)-\widetilde{g}^n(\textbf{t},0)| \leq \|\widetilde{g}^n(\textbf{t},.)\| |x|$$
and $\|\widetilde{g}^n(\textbf{t},.)\| \to 0$ as $n \to +\infty$. Thus, for almost every $\textbf{t} \in \mathcal{S}$,
$$\lim_{n \to + \infty}\widetilde{g}_{\textbf{t}} \circ \widetilde{g}_{\xi^{-1}(\textbf{t})} \circ \ldots \circ \widetilde{g}_{\xi^{-n}(\textbf{t})}(X)= \lim_{n \to + \infty}\widetilde{g}_{\textbf{t}} \circ \widetilde{g}_{\xi^{-1}(\textbf{t})} \circ \ldots \circ \widetilde{g}_{\xi^{-n}(\textbf{t})}(0).$$
Hence, $\widetilde{\gamma}$ induces a bony graph for $\widetilde{G}$. Since $\widetilde{\Delta} \cap X_\textbf{t}=\bigcap_{n\geq 0} X(\textbf{t},n)$, where $X_\textbf{t}:=\{\textbf{t}\}\times X$, $X(\textbf{t},n):=\widetilde{g}_{\textbf{t}} \circ \widetilde{g}_{\xi^{-1}(\textbf{t})}\circ \ldots \circ \widetilde{g}_{\xi^{-n}(\textbf{t})}(X)$ and by the construction this bony graph is equal to $\widetilde{\Delta}$.
By Theorem \ref{thm000}, the graph $\widetilde{\Gamma}$ is the support of an invariant ergodic measure.
It remains to prove that the bony graph is upper semicontinuous. It is enough to show that $\widetilde{\Delta}_{\textbf{t}}:=\widetilde{\Delta} \cap X_\textbf{t}$ is upper semicontinuous. Let $\textbf{t} \in \mathcal{S}$ and $\varepsilon > 0$ be given.
If $n$ is big enough then
$$\widetilde{g}_{\textbf{t}} \circ \widetilde{g}_{\xi^{-1}(\textbf{t})} \circ \ldots \circ \widetilde{g}_{\xi^{-n}(\textbf{t})}(X) \subset U_{\frac{\varepsilon}{2}}\widetilde{\Delta}_{\textbf{t}}.$$
Let $\textbf{t}^{\prime}$ be sufficiently close to $\textbf{t}$. Then $\widetilde{g}_{\textbf{t}^{\prime}} \circ \widetilde{g}_{\xi^{-1}(\textbf{t}^{\prime})} \circ \ldots \circ \widetilde{g}_{\xi^{-n}(\textbf{t}^{\prime})}$ is $C^2$ close to $\widetilde{g}_{\textbf{t}} \circ \widetilde{g}_{\xi^{-1}(\textbf{t})} \circ \ldots \circ \widetilde{g}_{\xi^{-n}(\textbf{t})}$.
Let $x \in X$ and $L_x$ be the line segment through the origin and connects the point $x$ to $-x$.
 Since $X$ is convex so $X=\cup_{x \in X}L_x$.
Now by using the Mean-value Theorem, for each $x \in X$ one has
$$\widetilde{g}_{\textbf{t}^{\prime}} \circ \widetilde{g}_{\xi^{-1}(\textbf{t}^{\prime})} \circ \ldots \circ \widetilde{g}_{\xi^{-n}(\textbf{t}^{\prime})}(L_x)\subset U_{\frac{\varepsilon}{2}}(\widetilde{g}_{\textbf{t}} \circ \widetilde{g}_{\xi^{-1}(\textbf{t})} \circ \ldots \circ \widetilde{g}_{\xi^{-n}(\textbf{t})}(L_x))$$
if $\widetilde{g}_{\textbf{t}^{\prime}} \circ \widetilde{g}_{\xi^{-1}(\textbf{t}^{\prime})} \circ \ldots \circ \widetilde{g}_{\xi^{-n}(\textbf{t}^{\prime})}$ is sufficiently $C^2$ close to $\widetilde{g}_{\textbf{t}} \circ \widetilde{g}_{\xi^{-1}(\textbf{t})} \circ \ldots \circ \widetilde{g}_{\xi^{-n}(\textbf{t})}$.
Therefore $$\widetilde{\Delta}_{\textbf{t}^{\prime}}\subset \widetilde{g}_{\textbf{t}^{\prime}} \circ \widetilde{g}_{\xi^{-1}(\textbf{t}^{\prime})} \circ \ldots \circ \widetilde{g}_{\xi^{-n}(\textbf{t}^{\prime})}(X)\subset U_{\frac{\varepsilon}{2}}(\widetilde{g}_{\textbf{t}} \circ \widetilde{g}_{\xi^{-1}(\textbf{t})} \circ \ldots \circ \widetilde{g}_{\xi^{-n}(\textbf{t})}(X))\subset U_{\varepsilon}(\widetilde{\Delta}_{\textbf{t}}).$$
Hence, the bony graph is upper semicontinuous and $(1)$ is satisfied.

The proof of the other statements are completely similar to the proof Theorem \ref{thm22} so it is omitted.
\end{proof}
Consider the projection $\widehat{\mu}=(\pi \times id)_{\ast}\widetilde{\mu}$, where
$\pi: \mathcal{S} \to S^1$ is defined by (\ref{e99}). Since $\widetilde{\mu}$ is an SRB measure with $\textnormal{supp}(\widetilde{\mu})=\widetilde{\Gamma}$ so $\widehat{\mu}$ is also an SRB measure for $\widetilde{F}$ with $\textnormal{supp}(\widehat{\mu})=(\pi \times id)\widetilde{\Gamma}$.
Hence the next theorem follows immediately.
\begin{theorem}
There exists a nonempty open set in $\mathcal{C}(\mathbb{T})$ such that any skew product $F$
from this set admits an attracting invariant set $\widehat{\Delta}$ for which the following statements hold:
\begin{enumerate}
\item $\widehat{\Delta}$ is the $(\pi \times id)$-image of a continuous bony graph.
    \item $\widehat{\Delta}$ is attracting in the sense that for almost all $t \in S^1$ and for every $x \in X$, $\lim_{n \to \infty}|\pi_x \widehat{\Delta} - \pi_x F^n(t,x)|=0$.
       \item $\widehat{\Delta}$ is the support of an ergodic invariant measure $\widehat{\mu}$.
  \item $(F, \widehat{\Delta}, \widehat{\mu})$ is Bernoulli and therefore it is mixing.
  \item $\widehat{\mu}$ is a unique $\textnormal{SRB}$ measure of $F$.
\end{enumerate}
\end{theorem}
Now Theorem A follows by the previous theorem.
To prove Theorem B, we apply Gorodetski-Ilyashenko-Negut techniques \cite{IN}.
By construction the skew product $F_\Psi$ and each perturbed skew product $\widetilde{F} \in \mathcal{C}(\mathbb{T})$ that is close enough to $F_\Psi$ satisfy the following
 \emph{modified dominated splitting condition} \cite[Def.~2]{IN}
$$ \max(\frac{1}{k}+\| \frac{\partial f_t^{\pm}}{\partial t}\|_{C_0},\| \frac{\partial f_t^{\pm}}{\partial x}\|_{C_0} )=L < k.$$
Consider any map $\mathcal{F}\in \mathcal{E}(\mathbb{T})_k$ which is $C^1$-close to $F_\Psi$.
Notice that though $F_\Psi$ is a skew product, $\mathcal{F}$ is usually not.
However, for small enough perturbations of $F_\Psi$, we are able to rectify the perturbed map and recover the structure.
By \cite[Thm.~5]{IK}, \cite[Thm.~9.2]{V} and \cite[Thm.~9.3]{V}, we will provide a solenoidal skew product $F$, $C^1$-close to $F_\Psi$, which is conjugated to the endomorphism $\mathcal{F}$.
These facts together with Theorem A and the argument used in Section 9 of \cite{V} imply Theorem B.
\section*{Acknowledgements}
During the preparation of this article the second author was partially
supported by grant from IPM (No. 94370127). He also thanks ICTP
for supporting through the association schedule.

\end{document}